\documentclass[10pt,a4paper,reqno]{amsart}
\usepackage{amsmath,amssymb,amsthm,mathtools,dsfont,bm,bbm,a4,xfrac} 

\usepackage{color}
\usepackage{wasysym}
\newtheorem{lem}{Lemma}[section]
\newtheorem{prop}{Proposition}[section]
\newtheorem{cor}{Corollary}[section]

\newtheorem{thm}{Theorem}[section]

\theoremstyle{definition}

\theoremstyle{remark}

\newtheorem{remark}{Remark}[section]
\newtheorem{remarks}{Remarks}[section]
\newtheorem*{remarks*}{Remarks}
\newtheorem*{remark*}{Remark}
\numberwithin{equation}{section}

\newcommand{\C}{\mathbb{C}}
\newcommand{\N}{\mathbb{N}}

\newcommand{\T}{\mathbb{T}}
\newcommand{\R}{\mathbb{R}}

\newcommand{\RR}{\mathcal{R}}
\newcommand{\RRw}{\mathcal{R}_{\tilde{\omega}}}
\newcommand{\ov}{\overline}

\newcommand{\norm}[1]{\|#1 \|}

\newcommand{\pt}{\partial}

\newcommand{\e}{\varepsilon}

\newcommand{\be}{\begin{equation}}
\newcommand{\ee}{\end{equation}}

\newcommand{\eps}{\varepsilon}

\newcommand{\Ms}{\mathsf{M}}
\newcommand{\Vs}{\mathsf{V}}
\newcommand{\Ws}{\mathsf{W}}
\newcommand{\As}{\mathsf{A}}
\newcommand{\Bs}{\mathsf{B}}

\newcommand{\weakto}{\rightharpoonup}
\newcommand{\ii}{i}

\renewcommand{\rho}{\varrho}

\newcommand{\Hil}{\mathsf{H}}
   
\newcommand{\inner}[2]{  \langle #1, #2  \rangle } 
\newcommand{\bbinner}[2]{ \bigg \langle #1, #2 \bigg \rangle } 
\newcommand{\comment}[1]{}

\catcode`@=11
\def\section{\@startsection{section}{1}%
  \z@{1.5\linespacing\@plus\linespacing}{.5\linespacing}%
  {\normalfont\bfseries\large\centering}}
\catcode`@=12

\begin{document}
\title[Calogero--Moser DNLS]{The Calogero--Moser derivative \\ nonlinear Schr\"odinger equation}

\author{Patrick G\'erard}
\address{P. G\'erard,  Laboratoire de Math\'ematiques d'Orsay, CNRS, Universit\'e Paris-Saclay, 91405 Orsay, France.}%
\email{patrick.gerard@universite-paris-saclay.fr}

\author{Enno Lenzmann}
\address{E. Lenzmann, University of Basel, Department of Mathematics and Computer Science, Spiegelgasse 1, CH-4051 Basel, Switzerland.}%
\email{enno.lenzmann@unibas.ch}


\date{\today}

\maketitle

\begin{abstract}
We study the Calogero--Moser derivative NLS equation
$$
\ii \pt_t u +\pt_{xx} u + (D+|D|)(|u|^2) u =0
$$
posed on the Hardy-Sobolev space $H^s_+(\R)$ with suitable $s>0$. By using a Lax pair structure for this $L^2$-critical equation, we prove global well-posedness for $s \geq 1$ and initial data with sub-critical or critical $L^2$-mass $\| u_0 \|_{L^2}^2 \leq 2 \pi$. Moreover, we prove uniqueness of ground states and also classify all traveling solitary waves. Finally, we study in detail the class of multi-soliton solutions $u(t)$ and we prove that they exhibit energy cascades in the following strong sense such that $\|u(t)\|_{H^s} \sim_s |t|^{2s}$ as $t \to \pm \infty$ for every $s > 0$.
\end{abstract}

\setcounter{tocdepth}{1}
\tableofcontents

\section{Introduction and Main Results}

\label{sec:intro}

This paper is devoted to the study of the {\em Calogero--Moser derivative nonlinear Schr\"odinger equation}, which can be written as
\be \tag{CM-DNLS} \label{eq:NLS}
\ii \pt_t u +\pt_{xx} u +(D+|D|)(|u|^2) u=0
\ee
for $u :I \times \R \to \C$ with some time interval $I \subset \R$. Here and in what follows, we use the standard notation $D=-i\pt_x$ and hence $|D|$ denotes the Fourier multiplier with symbol $|\xi |$.

We remark that equation \eqref{eq:NLS} was introduced in \cite{AbBeWi-09} as a formal continuum limit of classical Calogero--Moser systems \cite{Mo-75, OlPe-76}. Also, prior to \cite{AbBeWi-09}, a defocusing version given by
\be \tag{INLS}
\ii \pt_t u +\pt_{xx} u - (D+|D|)(|u|^2)  = 0 
\ee
was introduced in \cite{PeGr-95} under the name {\em intermediate nonlinear Schr\"odinger equation (INLS)}, as describing envelope waves in a deep stratified fluid. We will concentrate on \eqref{eq:NLS}, because it offers richer dynamics, e.\,g.~multi-soliton solutions with turbulence in Sobolev norms (see Theorem \ref{thm:growth_intro} below). However, part of our results can be extended to this defocusing version above.

\subsection{Symmetries, Phase Space, and Hamiltonian Features}
We observe that equation \eqref{eq:NLS} admits the invariance by phase, scaling and translation,
$$u(t,x)\mapsto \mathrm{e}^{i\theta} \lambda^{1/2}u(\lambda^2t, \lambda x+x_0), \quad x_0\in \R, \theta \in \R, \lambda >0,$$
which makes it a $L^2$-critical equation on the line. It also enjoys the Galilean invariance 
$$ 
u(t,x)\mapsto \mathrm{e}^{i\eta x-it\eta ^2}u(t,x-2t\eta ), \quad \eta \in \R,
$$
as well as the pseudo--conformal symmetry found by Ginibre and Velo for the $L^2$-critical NLS. Recall that a special case of this space-time transform reads
$$ 
u(t,x)\mapsto \frac{1}{t^{1/2}}\mathrm{e}^{i\frac{x^2}{4t}}u\left (-\frac1t, \frac xt\right ).
$$

In what follows, we are interested in solutions of \eqref{eq:NLS} satisfying the additional condition that
$$u(t)\in H^s_+(\R ):=\{ f\in H^s(\R ) : \mathrm{supp} (\hat f)\subset [0,+\infty )\},$$
where $H^s(\R )$ denotes the usual Sobolev space based on $L^2(\R)$. The spaces $H^s_+(\R)$ will serve as phase spaces on which we study \eqref{eq:NLS} as a Hamiltonian system.  

Recall that $H^0_+(\R )=L^2_+(\R )$ denotes the Hardy space of holomorphic functions on the complex upper half-plane. If we let $\Pi_+ : L^2(\R) \to L^2_+(\R)$ denote Cauchy--Szeg\H{o} orthogonal projection onto $L^2_+(\R)$ given by  
$$
\Pi_+(f)(x) := \frac{1}{2\pi} \int_0^\infty e^{i\xi x} \hat{f}(\xi) \, d\xi ,
$$
then equation \eqref{eq:NLS} can be written as
$$
\ii \pt_t u +\pt_{xx} u +2D_+(|u|^2) u=0.
$$
Here $D_+:=D\Pi_+$ can be seen as the compression of $D=-i\pt_x$ onto the Hardy space $L^2_+(\R)$. The positive Fourier frequency condition $\mathrm{supp} (\hat f)\subset [0,+\infty )$ is interpreted as a chirality condition in \cite{AbBeWi-09}. In fact, such a condition naturally appears  if one thinks of the Benjamin--Ono equation,
\be \label{eq:BO} \tag{BO}
\pt_tv+\pt_x|D|v-\partial_x(v^2)=0,
\ee
which is known to be well-posed  for real valued functions $v$; see \cite{MoPi-12, KiLaVi-23}. Introducing the new unknown $u=\Pi_+v$, the condition $v=\overline v$ reads $v=u+\overline u$, so that (BO) is equivalent to
$$
\ii \pt_t u+\partial_{xx} u+D(u^2)+2D_+(|u|^2)=0.
$$
This way, \eqref{eq:NLS} and its defocusing sibling can be seen as $L^2$-critical versions of (BO). 

Notice that the pseudo-conformal symmetry does not preserve chirality and that the Galilean transformation acts on chiral solutions of \eqref{eq:NLS} only if $\eta \geq 0$.\\
 
Now, let us come to the Hamiltonian properties of \eqref{eq:NLS}. To this end, we introduce the following gauge transformation
 \be\label{def:gauge}
 v(x):=u(x)\, \mathrm{e}^{-\frac{i}{2} \int_{-\infty}^x |u(y)|^2 \, dy},
 \ee
 which turns out to be a diffeomorphism of $H^s(\R )$ into itself for every $s\geq 0$. An elementary calculation shows that \eqref{eq:NLS} is equivalent to
 the equation 
 \be \label{eq:QDNLS}
 \ii\pt_tv+\pt_{xx} v+|D|(|v|^2)v-\frac 14 |v|^4v=0.
 \ee
This is a Hamiltonian PDE with the standard symplectic form $\omega(h_1,h_2) = \mathrm{Im}\inner {h_1}{h_2} _{L^2}$ and the energy functional
$$
\widetilde E(v):=\frac 12 \Vert \partial_xv\Vert_{L^2} ^2-\frac 14 \inner{|D|(|v|^2)}{|v|^2}_{L^2}+\frac 1{24}\Vert v\Vert_{L^6}^6.
$$
By classical product identities for the Hilbert transform $\Hil$ (see Appendix \ref{app:misc} for details), the energy functional $\widetilde E$ can be written as 
$$
\widetilde E(v) = \frac{1}{2} \int_\R |\pt_x v + \frac{1}{2} \Hil(|v|^2) v|^2 \, dx \geq 0.
$$
Inverting the gauge transformation \eqref{def:gauge} and in view of $\Pi_+ = \frac{1}{2}(1+i \Hil)$, we find that $E(u) = \widetilde E(v)$ is an energy functional for \eqref{eq:NLS} given by
\be \label{def:E}
E(u) = \frac{1}{2} \int_{\R} |\pt_x u - i \Pi_+(|u|^2) u|^2 \, dx .
\ee

In summary, we deduce that \eqref{eq:NLS} is a Hamiltonian equation generated by the energy functional $E(u)$ and the symplectic form 
$$ 
\omega_u^\sharp(h_1,h_2):=\mathrm{Im}\inner{h_1}{h_2}_{L^2}+\int \!\!\int_{\R\times \R}\mathrm{Re}(\ov uh_1)(x)\mathrm{Re}(\ov uh_2)(y)\mathrm{sgn}(x-y)\, dx\, dy ,
$$
which is just the pullback of the standard symplectic $\omega$ under the gauge transformation $u \mapsto v$ defined in \eqref{def:gauge}. Recall that we will study \eqref{eq:NLS} as Hamiltonian PDE on the phase spaces $H^s_+(\R)$ corresponding to chiral solutions. It is interesting to note that $\omega_u^\sharp$ provides a non-standard symplectic form on the spaces $H^s_+(\R)$ with $s \geq 0$.\\

Next, we discuss the conservation laws exhibited by \eqref{eq:NLS}. Due to symmetry by complex phase shifts, spatial translations and its Hamiltonian nature, we easily obtain the following conserved quantities:
$$
    M(u) = \int_\R |u|^2 \, dx \quad \mbox{($L^2$-mass)}, \qquad P(u) = \int_\R (D u \ov{u} - \frac{1}{2} |u|^4 ) \, dx \quad \mbox{(Momentum)},
$$
$$
E(u) =  \frac{1}{2} \int_{\R} |\pt_x u - i \Pi_+(|u|^2) u|^2 \, dx \quad \mbox{(Energy)}.
$$
In the expression for the conserved momentum $P(u)$, the nonlinear term $|u|^4$ arises due to the non-standard symplectic structure $\omega_u^\sharp$. In fact, we will show below that  $M(u), P(u)$ and $E(u)$ belong to an {\em infinite hierarchy of conservation laws} $\{ I_k(u) \}_{k=0}^\infty$ based on a Lax pair structure for \eqref{eq:NLS}; see Section \ref{sec:WP}.\\

Finally, we briefly comment on the $L^2$-critical nature of \eqref{eq:NLS}. As one may expect, there exists a special solution which separates solutions into small and large data. Indeed, we will prove that the energy $E(u)$ has a {\em unique} (up to symmetries) minimizer given by the rational function
$$
\RR(x) = \frac{\sqrt{2}}{x+i} \in H^1_+(\R),
$$
which we refer to as the \textbf{ground state} for \eqref{eq:NLS}; see Section \ref{sec:gs}. An elementary calculation shows that $u(t,x)=\RR(x)$ provides a {\em static solution} of \eqref{eq:NLS} and its $L^2$-mass is found to be
$$
M(\RR) = \int_\R \frac{2}{1+x^2} \, dx  = 2\pi.
$$
As we will see below, this number $2 \pi$ provides a threshold in the analysis of \eqref{eq:NLS}. Consequently, we shall refer to solutions $u(t) \in H^s_+(\R)$ with 
$$
M(u_0) < M(\RR), \quad M(u_0) = M(\RR), \quad M(u_0) > M(\RR)
$$
as having sub-critical, critical, and super-critical $L^2$-masses, respectively. The main results of this paper will address these various regimes.

\subsection{Main Results}

As a staring point, we first establish local well-posedness of \eqref{eq:NLS} for initial data in $H^s_+(\R)$ with $s > \sfrac{1}{2}$. For sufficiently regular initial data in $H^s_+(\R)$ with $s > \sfrac{3}{2}$, this follows from Kato's classical iteration scheme for quasilinear evolution equations. Extending the local well-posedness to less regular data in $H^{s}_+(\R)$ with $\sfrac{1}{2} < s \leq \sfrac{3}{2}$ can then be achieved by adapting arguments from \cite{MoPi-10}, which in turn is inspired by Tao's gauge trick for the Benjamin-Ono equation \cite{Ta-04}. However, for the rest of the paper, we will be mainly be concerned with solutions of \eqref{eq:NLS} such that $u(t) \in H^s_+(\R)$ with some integer $s \geq 1$.

The general question of global well-posedness for \eqref{eq:NLS} seems to be rather delicate because of the focusing $L^2$-criticality of the nonlinearity, which might generate blowup of solutions in finite time. The following result establishes global well-posedness for initial data with finite energy and $L^2$-mass that is less or equal to the ground state mass.

\begin{thm}[Global Well-Posedness Result] \label{thm:gwp} Let $s \geq 1$ be an integer. Then \eqref{eq:NLS} is globally well-posed for initial data in $u_0 \in H^s_+(\R)$ with $L^2$-mass
$$
M(u_0) \leq M(\RR) = 2 \pi.
$$
Moreover, we have the a-priori bound
$$
\sup_{t \in \R} \|u(t)\|_{H^s} < +\infty,
$$
provided the strict inequality $M(u_0) < M(\RR)$ holds.
\end{thm} 

\begin{remarks}
1) In the case of sub-critical $L^2$-mass, the a-priori bounds on $\| u(t) \|_{H^s}$ will follow from exploiting an infinite hierarchy of conservation laws for \eqref{eq:NLS}. We refer to Section \ref{sec:lax} for a detailed discussion.

2) The case of critical $L^2$-mass when $M(u_0) = M(\RR)$ is rather delicate to handle and will follow from ruling out the so-called minimal mass blowup solutions for \eqref{eq:NLS} with finite energy. A key element in the proof will be the slow algebraic decay of the ground states $\RR$. 

3) It is an interesting open question whether global-in-time existence holds for large initial data in $H^s_+(\R)$ with $s \geq 1$ and $L^2$-mass $M(u_0) > M(\RR)$. As a striking example below, there exist smooth global-in-time solutions for \eqref{eq:NLS} given by multi-solitons, which always blowup in infinite time due to unbounded growth of all Sobolev norms $\| u(t) \|_{H^s}$ for any $s>0$. 

4) By applying the pseudo-conformal transformation to the static solution $\RR(x)$, we obtain the explicit solution
$$
u_{\mathrm{sing}}(t,x) = \frac{1}{t^{1/2}} e^{i x^2/4t} \RR \left ( \frac{x}{t} \right ) \in L^2(\R) \quad \mbox{for all $t>0$},
$$
which solves \eqref{eq:NLS} and becomes singular as $t \to 0^-$. Due to slow algebraic decay of $\RR(x)$, we find that $u_{\mathrm{sing}}(t) \not \in H^1(\R)$ has no finite energy\footnote{A closer inspection shows that $u_{\mathrm{sing}}(t,\cdot) \in H^s(\R)$ for all $0 \leq s < \sfrac{1}{2}$.} and, moreover, we see that the solution $u_{\mathrm{sing}}(t) \not \in L^2_+(\R)$ fails to be chiral. Still, this explicit example shows that we cannot expect global well-posedness for \eqref{eq:NLS} with arbitrary initial data in the scaling-critical space $L^2(\R)$. It remains an intriguing open question if initial data in $L^2_+(\R)$ will always lead to global-in-time solutions for \eqref{eq:NLS}.
\end{remarks}

Next, we turn our attention to sufficiently regular solutions $u(t)\in H^s_+(\R)$ of \eqref{eq:NLS} with initial data having critical or super-critical $L^2$-mass:
$$
M(u_0) \geq M(\RR) = 2\pi.
$$
In this regime of sufficiently large data, we expect \eqref{eq:NLS} to possess traveling ground state solitons as well as multi-soliton solutions. As a main result in this setting, we completely classify all traveling solitary waves for the Calogero--Moser DNLS with finite energy by showing that are given by the ground state $\RR(x)$ up to scaling, phase, translation, and Galilean boosts prserving the chirality condition.

\begin{thm}[Classification of Traveling Solitary Waves] \label{thm:trav_wave}
Every traveling solitary wave for equation \eqref{eq:NLS} in $H^1_+(\R)$ is of the form
$$
u(t,x) = e^{\ii \theta + \ii \eta x - \ii \eta^2 t} \lambda^{1/2} \RR(\lambda (x-2 \eta t) + y)
$$
with some $\theta \in [0, 2\pi), y \in \R, \lambda > 0$, and $\eta \geq 0$. 

In particular, every traveling solitary waves $u(t) \in H^1_+(\R)$ for \eqref{eq:NLS} have critical $L^2$-mass $M(u) = M(\RR)$.
\end{thm}

\begin{remarks}
1) The condition $\eta \geq 0$ enters through the chirality condition $u(t) \in H^1_+(\R)$ and thus traveling solitary waves can only move to right. If we take negative values $\eta <0$ above, we obtain left-moving traveling solitary waves $u(t) \in H^1(\R)$ solving \eqref{eq:NLS}; see Section \ref{sec:gs} for the definition of traveling solitary waves.

2) A key step in the complete classification above is to establish {\em uniqueness} of (non-trivial) minimizers of the energy $E(u)$, which is equivalent to classifying all solutions $u \in H^1_+(\R)$ of the nonlinear equation
$$
D u - \Pi_+(|u|^2) u = 0. 
$$
We refer to Section \ref{sec:gs} below for details including a more general result assuming only that $u \in H^1(\R)$.  
\end{remarks}

As our final main result, we study the dynamics of {\em multi-soliton solutions} for \eqref{eq:NLS}; see Section \ref{sec:dyn_sol} below for a precise definition using the Lax pair structure. For the Calogero--Moser DNLS, it turns out that multi-solitons $u=u(t,x)$ are rational functions of $x \in \R$ in the Hardy spaces $L^2_+(\R)$. As an interesting fact, we remark that they necessarily have a {\em quantized $L^2$-mass} given by
$$
M(u) = 2 \pi N \quad \mbox{with $N=1,2,3,\ldots$}
$$
In the special case when $N=1$, the multi-solitons are given by the ground state $\RR(x)$ up to symmetries. For $N \geq 2$, we note that multi-solitons have super-critical $L^2$-mass. As a consequence, the proof of their global-in-time existence is far from trivial and will follow from the analysis of a suitable inverse spectral formula based on the Lax structure. As an outcome, we obtain a detailed dynamical description in the long-time limit. Here, a surprising feature is the general `turbulent' behavior of multi-solitons with $N \geq 2$, leading to unbounded growth of higher Sobolev norms (energy cascades) as follows.      

\begin{thm}[Growth of Sobolev Norms] \label{thm:growth_intro}
For every $N \geq 2$, every multi-soliton $u=u(t,x)$ for \eqref{eq:NLS} exists for all times $t \in \R$ and it exhibits growth of Sobolev norms such that
$$
\| u(t) \|_{H^s} \sim_s |t|^{2s} \quad \mbox{as} \quad t \to \pm \infty,
$$
for any real number $s> 0$.
\end{thm}

\begin{remarks*}
1) In Section \ref{sec:dyn_sol}, we make a detailed analysis of the dynamics of multi-solitons. After having established their global-in-time existence, we show that there exists a sufficiently large time $T=T(u_0) \gg 1$ such that a multi-soliton reads
\be \label{eq:multi_sol}
u(t,x) = \sum_{k=1}^N \frac{a_k(t)}{x-z_k(t)} \quad \mbox{for $t \geq T$},
\ee
with residues $a_1(t), \ldots, a_N(t) \in \C$ and pairwise distinct poles $z_1(t),\ldots, z_N(t) \in \C_-$ that satisfy a complexified version of the rational Calogero--Moser system for $N$ classical particles. A detailed investigation (exploiting on the Lax pair structure) then yields that the poles -- except for $z_1(t)$ -- will all approach the real axis asymptotically, i.\,e., 
$$
\mathrm{Im} \, z_k(t) \to 0 \quad \mbox{as} \quad  t \to +\infty  \quad \mbox{for $2 \leq k \leq N$}. 
$$
A careful analysis of this fact then leads to the precise growth bound in Theorem \ref{thm:growth_intro}. The limit $t \to -\infty$ can be handled in the same way.  

2) It is a subtle fact that multi-solitons $u(t,x)$ may fail to be of the form \eqref{eq:multi_sol} for {\em all} times $t\in \R$. That is, we can have collisions of poles in finite time, which renders the form \eqref{eq:multi_sol} invalid. To handle this collision scenario (see explicit examples for $N=2$ in Section \ref{sec:dyn_sol}), we will make use of a general representation formula of $u(t,x)$ in terms of an inverse spectral formula. See Section \ref{sec:dyn_sol} for details.   

3) It is an interesting open question whether the growth phenomenon in Theorem \ref{thm:growth_intro} is stable under perturbations of multi-solitons. 
\end{remarks*}

\subsection{Comments on the Lax Structure}
A central feature of \eqref{eq:NLS} is the fact that it admits a Lax pair. That is, as detailed in Section \ref{sec:WP}, we can recast the dynamical evolution into commutator form  
\be \label{eq:Lax_intro}
\frac{d}{dt} L_u = [B_u, L_u],
\ee
where the Lax operator $L_u$ is given by
\be
L_u = D - T_u T_{\bar{u}}.
\ee
This defines an unbounded self-adjoint operator acting the Hardy space $L^2_+(\R)$ with a suitable operator domain, depending on the regularity of $u$. Here $T_b(f) = \Pi_+(bf)$ denotes the Toeplitz operator on $L^2_+(\R)$ with symbol $b$. 

As an important consequence of \eqref{eq:Lax_intro}, we will find an {\em infinite hierarchy of conservation laws} in terms of expressions
$$
I_k(u) = \langle L^k_u u, u \rangle \quad \mbox{with} \quad k=0,1,2,\ldots
$$
provided that $u(t) \in H^s_+(\R)$ is a sufficiently regular solution of \eqref{eq:NLS}. It is an intriguing feature, due to the $L^2$-criticality of the problem, that the hierarchy $\{ I_k(u) \}_{k \in \N}$ will generally provide a-priori bounds on solutions only if we have sub-critical $L^2$-mass; see Theorem \ref{thm:gwp}. This is in striking contrast to many other completely integrable PDEs (e.\,g.~KdV, Benjamin-Ono, cubic NLS etc.) where the corresponding hierarchy of conservation laws yields control over any smooth solutions without any assumption on its size.\\

In Section \ref{sec:lax}, we examine the spectral properties of the Lax operator $L_u$ in more detail. Based on a key commutator formula, we find a sharp bound on the number of eigenvalues $N$ of the form
\be \label{ineq:N}
N \leq \frac{\|u \|_{L^2}^2}{2 \pi}.
\ee
Furthermore, we prove that every eigenvalue of $L_u$ is simple. As interesting aside, we remark that this bound not only applies to isolated eigenvalues, but also to eigenvalues which are embedded in the essential spectrum $\sigma_{\mathrm{ess}}(L) = [0,\infty)$. Moreover, we emphasize the fact that we can easily generate embedded eigenvalues of $L_u$ by action of the Beurling--Lax semigroup $\{ e^{i \eta x} \}_{\eta \geq 0}$ acting on $L^2_+(\R)$; see Section \ref{sec:lax} again.

In terms of spectral theory, it is a natural question to study which potentials $u \in L^2_+(\R)$ will lead to equality in the general bound \eqref{ineq:N}. Here we will find a distinguished class of potentials given by rational functions of the form
$$
u(x) = \frac{P(x)}{Q(x)} \in H^1_+(\R),
$$
where $Q,P \in \C[x]$ are suitable polynomials with $\deg Q = N$ and $\deg P \leq N-1$; see Proposition \ref{prop:spectralNsoliton}. We will refer to these $u(x)=P(x)/Q(x)$ as above as {\em multi-soliton potentials} and the corresponding solution will called {\em multi-solitons} for \eqref{eq:NLS}. As an immediate consequence of saturating the bound \eqref{ineq:N}, we obtain the multi-solitons $u(t,x)$ have quantized $L^2$-mass with 
$$
M(u) = 2 \pi N.
$$

Another noteworthy feature of any multi-soliton solution $u(t,x)$ is that it is {\em completely} supported in the pure point spectrum of the Lax operator, i.\,e, we have
$$
u(t) \in \mathcal{E}_{pp}(L_{u(t)}),
$$
where $\mathcal{E}_{pp}$ denotes the $N$-dimensional space spanned the eigenfunctions of $L_u$. This fact will allow us to derive a very explicit inverse spectral formula representing a multi-soliton. This will enable us to prove global-in-time existence and, more strikingly, the growth bounds states in Theorem \ref{thm:growth_intro}. In the future, we plan to further refine the spectral analysis of $L_u$ in order to study the long-time behavior of solutions of \eqref{eq:NLS} beyond the case of multi-solitons.\\

Finally, we remark that the Lax structure for \eqref{eq:NLS} bears some resemblance to the Lax structure for (BO), which is known to have the Lax operator 
$$
L_u^{\mathrm{(BO)}} = D - T_u
$$
acting on the Hardy space $L^2_+(\R)$; see e.\,g.~\cite{GeKa-21}. Note that the occurrence of $T_u$ in $L_u^{\mathrm{(BO)}}$ instead of $T_u T_{\bar{u}}$ in $L_u$ is consistent with the different degrees of the nonlinearity (quadratic vs.~cubic).

\subsection{Comparison to other PDEs}

Let us comment on the energy cascade phenomenon in Theorem \ref{thm:growth_intro} in comparison to other Hamiltonian PDEs on the line. 
Among the recently studied Hamiltonian PDEs on the line, the closest  one to \eqref{eq:NLS} is certainly the derivative nonlinear Schr\"odinger equation,
\be\tag{DNLS}\label{eq:DNLS}
i\pt_tu+\pt_{xx}u+i\pt_x(|u|^2u)=0 ,
\ee
which -- like \eqref{eq:NLS} --  is $L^2$-mass critical with a Lax pair structure. Using the Lax pair structure, global existence was first proved in \cite{JeLiPeSu-20} in the space $H^2(\R)\cap \hat H^2(\R)$. Then global well-posedness with uniform bounds in $H^s$, $s\geq 1/2$ was obtained in \cite{BaPe-22}, \cite{BaLePe-21}. Quite recently, the flow map was extended to the whole of $L^2(\R )$ in \cite{HaKiNtVi-22},
proving that all trajectories of \eqref{eq:DNLS} are uniformly equicontinuous with values in $L^2(\R)$. All these results prevent any kind of energy cascade and are therefore in strong contrast with the dynamics of \eqref{eq:NLS}, which turns out to be much richer.

Another integrable Hamiltonian PDE on the line is the cubic Szeg\H{o} equation, see \cite{Po-11},
\be \label{eq:Szego}
i\pt_tu=\Pi_+(|u|^2u) ,
\ee
where multi-solitons were recently studied in \cite{GePu-22}, and where energy cascades were displayed  under some degeneracy assumption of the spectrum of the corresponding Lax operator. There is definitely some similarity in the  approaches to multi-solitons in \eqref{eq:Szego} and \eqref{eq:NLS}, particularly in the inverse spectral formulae. However, let us emphasize that the spectral properties of the Lax operators  are very different, and that energy cascades for 
multi-solitons in \eqref{eq:Szego}  only occur under some degeneracy assumption, while they always occur for multi-solitons in \eqref{eq:NLS}. This suggests that the dynamics of \eqref{eq:NLS} is particularly  turbulent, even compared to the non-dispersive equation \eqref{eq:Szego}. We hope to explore other aspects of this dynamics in the  near future.

\subsection{Notation} We denote by $\inner f g = \int_\R f \ov{g}$ the $L^2$-inner product of functions $f, g$ on the line.  We recall that $\Pi_+$ denotes the orthogonal projector from $L^2(\R )$ onto the Hardy space $L^2_+(\R )$. Sometimes, we will also use the notation $\Pi_-=1-\Pi_+$.
Notice that, for every $L^2$ function $f$, 
$$\Pi_-(f)=\ov{\Pi_+(\ov f)}\ .$$
Finally, observe that the Hilbert transform $\Hil =-i\, \mathrm{sgn}(D)$ is related to $\Pi_\pm$ by the identities
$$\Pi_+=\frac12(1+i\Hil )\ ,\ \Pi_-=\frac 12(1-i\Hil )\ .$$

\subsection*{Acknowledgments} E.\,L.~acknowledges financial support from the Swiss National Science Foundation (SNSF) under Grant No.~204121.

\section{Well-Posedness, Lax Structure, and Conservation Laws}\label{sec:WP}

In this section, we study the Cauchy problem for \eqref{eq:NLS} in $H^s_+(\R)$ with suitable $s$. As a key element for obtaining global-in-time solutions, we will find a Lax pair structure on the Hardy-type space $H^s_+$, which will generate an infinite hierarchy of conservation laws. With this at hand, we will derive a-priori bounds for initial $u_0 \in H^s_+(\R)$ with integer $s \geq 1$ and sub-critical $L^2$-mass $M(u_0) < M(\RR)$. 

\subsection{Local Well-Posedness}

As starting point for local well-posedness, we consider the case of initial data in $H^s_+(\R)$ with $s > \sfrac{3}{2}$, where Kato's classical iterative scheme for quasilinear evolution equations can be utilized. We remark that the presence of the derivative term $D_+(|u|^2) u$ raises some analytic challenges that need to be addressed. 

\begin{prop} \label{prop:lwp_kato}
Let $s > \sfrac{3}{2}$. For any $R >0$, there is some $T(R) > 0$ such that, for every $u_0 \in H^s_+(\R)$ with $\|u_0 \|_{H^s} \leq R$, there exists a unique solution  $u \in C([-T,T];H^s_+(\R))$ of \eqref{eq:NLS} with $u(0) = u_0$. 

Furthermore, the $H^\sigma$-regularity of $u_0$ for $\sigma >s$ is propagated on the whole maximal interval of existence of $u$, and the flow map $u_0\mapsto u(t)$ is continuous on $H^s$.
\end{prop}

\begin{proof}
As already mentioned above, we apply a Kato-type iterative scheme to obtain this result. For concreteness, we shall consider the case
$$
s=2
$$
in what follows. The general case $s > \sfrac{3}{2}$ can be handled in an analogous way.

We first write \eqref{eq:NLS} as
\be \label{eq:NLS_alt}
\pt_t u = i \pt_{xx} u + 2 T_u T_{\ov{u}} \pt_x u + 2 u H_u \pt_x u,
\ee
where
\be 
T_a f := \Pi_+(af), \quad H_b f := \Pi_+(b \ov{f})
\ee
denote the {\em Toeplitz} and {\em Hankel operators} acting on $L^2_+(\R)$ with symbols $a$ and $b$, respectively. Our first observation is that the term $H_u \pt_x u$ is of order 0 if $u$ is smooth enough.

\begin{lem} \label{lem:H_smooth}
If $u \in H^{\frac 3 2}_+(\R)$, then $H_u \pt_x : L^2_+(\R) \to L_+^2(\R)$ is bounded with  $\| H_u (\pt_x f) \|_{L^2} \leq \frac{1}{\sqrt{2\pi}}\| u \|_{\dot{H}^{3/2}} \|f \|_{L^2}$. If $u, v \in H^2_+(\R)$, then $H_u \pt_x v \in H^2_+(\R)$ with
$$
\| H_u \pt_x v \|_{H^2} \leq C \|u \|_{H^2} \|v \|_{H^2}
$$
with some constant $C>0$.
\end{lem}

\begin{proof}
Let $f \in H^1_+(\R)$. Then
$$
\widehat{H_u \pt_x f}(\xi) = -\int_0^\infty \widehat{u}(\xi+\eta) \ov{\widehat{f}(\eta)} \frac{d \eta}{2 \pi} \quad \mbox{for} \quad \xi \geq 0.
$$
Consequently,
\begin{align*}
|\widehat{H_u \pt_x f}(\xi)|^2 & \leq \left | \int_0^\infty |\widehat{u}(\xi+\eta)||\xi+\eta| |\widehat{f}(\eta)| \frac{d \eta}{2\pi} \right |^2 \\
& \leq \int_0^\infty |\widehat{u}(\xi+\eta)|^2 (\xi+\eta)^2 \frac{d \eta}{2 \pi}  \cdot \int_0^\infty |\widehat{f}(\eta)|^2 \frac{d \eta}{2\pi}. 
\end{align*}
Thus
\begin{align*}
\|H_u \pt_x f \|_{L^2}^2 & \leq \int_0^\infty \int_0^\infty |\widehat{u}(\xi+\eta)|^2 (\xi+\eta)^2 \frac{d\eta}{2\pi} \frac{d\xi}{2 \pi} \|f \|_{L^2} \\
&  \leq \int_0^\infty |\widehat{u}(\zeta)|^2 \zeta^3 \frac{d\zeta}{4 \pi^2} \|f \|_{L^2}^2 = \frac{1}{2\pi} \| u \|_{\dot{H}^{3/2}}^2 \|f\|_{L^2}^2.
\end{align*}
By density, this bound extends to all $f \in L^2_+(\R)$. This proves the first claim of the lemma.

For the second statement, we note that this follows from the first statement combined with Sobolev embeddings and the identity
$$
\pt_{xx}(H_u \pt_x v) = H_u\pt_x (\pt_{xx} v) + 2 H_{\pt_x u} \pt_x v + H_{\pt_{xx} u} \pt_x v.
$$
This completes the proof of Lemma \ref{lem:H_smooth}.
\end{proof}

In view of \eqref{eq:NLS_alt}, we consider the following iteration scheme
\be \label{eq:NLS_iter}
\pt_t u^{k+1} =i \pt_{xx} u^{k+1} + 2 T_{u^k} T_{\ov{u}^k} \pt_x u^{k+1} + 2 u^k H_{u^k} \pt_x u^k
\ee
with initial datum $u^{k+1}(0,x) =u_0(x) \in H^2_+(\R)$. Notice that $T_u T_{\ov{u}}$ is a self-adjoint operator. Hence a standard energy methods yields the following result.

\begin{lem} \label{lem:energy}
Let $u \in C([-T,T], H^2_+(\R))$ with some $T>0$, $p \in \{0,1,2\}$, and $w_0 \in H^p_+(\R)$, $f \in L^1([-T,T]; H^p_+(\R))$. Then there exists a unique $w \in C([-T,T]; H^p_+(\R))$ such that
$$
\pt_t w = i \pt_{xx} w+ 2 T_u T_{\ov{u}} w + f, \quad w(0,x) = w_0(x).
$$
Furthermore,
$$
\sup_{|t| \leq T} \|w(t) \|_{H^p} \leq C \mathrm{e}^{C \int_{-T}^T \|u(t)\|_{H^2}^2 \, dt} \left( \|w_0 \|_{H^p} + \| f \|_{L^1_t H^p} \right ).
$$
\end{lem}

Coming back to the scheme \eqref{eq:NLS_iter}, we see that Lemmas \ref{lem:H_smooth} and \ref{lem:energy} allow us to construct by induction a sequence $(u^k)$ in $C(\R; H^2_+(\R))$ with $u^0(t,x) = u_0(x)$. We are now going to prove that if $\|u_0 \|_{H^2} \leq R$ and $T(R)>0$ suitably chosen, then the sequence $(u^k)$ is bounded in $H^2_+(\R)$  and uniformly convergent in $L^2_+(\R)$ for $|t| \leq T(R)$.

Let us first prove that $(u^k)$ is bounded in $H^2_+(\R)$ for $|t| \leq T(R)$ with suitably chosen $T(R)>0$. Indeed, by using the second estimate in Lemma \ref{lem:H_smooth} together with the bound in Lemma \ref{lem:energy} for $p=2$, we obtain
$$
\sup_{|t| \leq T} \|u^{k+1} (t) \|_ {H^2} \leq C\mathrm{e}^{C \int_{-T}^T \|u^k(t) \|_{H^2}^2 \, dt} \left ( \|u_0 \|_{H^2} + \int_{-T}^T \|u^k(t)\|_{H^2}^3 \,dt \right ).
$$
Assume $\|u_0 \|_{H^2} \leq R$ and let $R_1 = (1+C) R$. Since $R_1 > CR$, we can choose $T=T(R) > 0$ such that
$$
C \mathrm{e}^{2CT R_1^2}(R+2TR_1^3) \leq R_1.
$$
By an elementary induction argument, we find that $\sup_{|t| \leq T(R)} \|u^k(t) \|_{H^2} \leq R_1$ for all $k$.

Next, we show that we have a contraction property of the sequence $(u^k)$ in $L^2_+(\R)$ for $|t| \leq T(R)$ as follows. Observe that
\begin{align*}
& \pt_t (u^{k+1}-u^k) = i \pt_{xx} (u^{k+1}-u^k) + 2 T_{u^k} T_{\ov{u}^k} \pt_x (u^{k+1}-u^k) + \\
& + 2 (T_{u^k} T_{\ov{u}k}- 2 T_{u^{k-1}} T_{\ov{u}^{k-1}}) \pt_x u^k + 2 u^k H_{u^k} \pt_x u^k - 2 u^{k-1} H_{u^{k-1}} \pt_x u^{k-1}. 
\end{align*}
Using the estimate of Lemma \ref{lem:energy} with $p=0$ and the bound on $u^k$ in $H^2$, we infer
$$
\sup_{|t| \leq T} \|u^{k+1}(t)-u^k(t)\|_{L^2} \leq K T \sup_{|t| \leq T} \|u^k(t) - u^{k-1}(t)\|_{L^2}
$$
with some constant $K>0$. If we choose $T=T(R)>0$ from above small enough to ensure that $KT < 1$, then the series $\sup_{|t| \leq T(R)} \|u^{k+1}(t) - u^k(t) \|_{L^2}$ is geometrically convergent. 

Finally, the sequence $(u^k)$ is uniformly weakly convergent in $C([-T,T]; H^2_+(\R))$ and strongly convergent in $C([-T,T]; L^2_+(\R))$. Hence its limit $u(t)$ solves \eqref{eq:NLS_alt} -- and therefore \eqref{eq:NLS}. (To prove that the limit $u(t)$ actually belongs to $C([-T,T]; H^2_+(\R))$, we can invoke Tao's frequency envelope method \cite{Ta-04} or adapt an argument due to Bona-Smith \cite{BoSm-75}.)

Uniqueness follows along the same lines as the contraction property in $L^2_+(\R)$. The proof of Proposition \ref{prop:lwp_kato} is now complete.
\end{proof}

Following the analysis in \cite{MoPi-10}, we can further lower the regularity for local well-posedness to initial data in $H^s_+(\R)$ with $s > \sfrac{1}{2}$.  In particular, we can reach the energy space $H^1_+(\R)$ for \eqref{eq:NLS}. In fact, the arguments adapt Tao's frequency localized gauge transform introduced to treat low regularity solutions for the Benjamin-Ono equation. For \eqref{eq:NLS}, we obtain the following local well-posedness result.

\begin{thm}[Local Well-Posedness in $H^s_+$ with $s > \sfrac{1}{2}$] \label{thm:lwp}
Let $u_0 \in H^s_+(\R)$ with some $s > \sfrac{1}{2}$. Then there exist a time $T=T(\|u_0\|_{H^s}) > 0$ and some Banach space $Z_{s,T} \subset C([-T,T]; H^s_+(\R))$ and a unique solution $u \in Z_{s,T}$ of \eqref{eq:NLS} with initial datum $u(0)=u_0$. 

Furthermore, the $H^\sigma$-regularity of $u_0$ for $\sigma >s$ is propagated on the whole maximal interval of existence of $u$, and the flow map $u_0\mapsto u(t)$ is continuous on $H^s$.
\end{thm}

\begin{remark*}
Recall that \eqref{eq:NLS} is $L^2$-critical with respect to scaling. It remains a fundamental open problem to understand the case $u \in H^s_+(\R)$ when $0 \leq s \leq \sfrac{1}{2}$.
\end{remark*}

\begin{proof}
We can adapt the estimates proven in \cite{MoPi-10} to our case. Suppose that $u_0 \in H^s_+(\R)$ with some $s > \sfrac{1}{2}$. For $\eps > 0$, let $\eta_\eps(x) = \sqrt{\frac{1}{\eps \pi}} \mathrm{e}^{-|x|^2/\eps}$  be a Gaussian mollifier. Then $u_{0,\eps}(x) = (\eta_\eps \ast u_0)(x)$ satisfies $u_{0,\eps} \in H^\infty_+(\R) \subset H^2_+(\R)$. By Proposition \ref{prop:lwp_kato}, there exists a unique solution $u_\eps \in C([-T_\eps, T_\eps]; H^2_+(\R))$ with $u_\eps(0) = u_{0, \eps}$. We can now apply the arguments in the proof of Theorem 1.1 in \cite{MoPi-10}. First, we can show that there exists $T=T(\|u_0 \|_{H^s}) > 0$ satisfying $T_\eps \geq T$ for all $\eps > 0$. Then following Proposition 3.2 in \cite{MoPi-10} we see that $(u_\eps)$ is Cauchy in $Z_{s,T}$ as $\eps \to 0$; we refer to \cite{MoPi-10} for the definition of the Banach space $Z_{s,T}$. Finally, the uniqueness of the limit of $(u_\eps)$ can be proven by the estimate (3.40) in \cite{MoPi-10}.
\end{proof}

\subsection{Lax Pair and Conservation Laws}

In this subsection, we will show that \eqref{eq:NLS} admits a {\em Lax pair} with certain densely defined operators $L_u$ and $B_u$ acting on the Hardy space $L^2_+(\R)$. Here we will exploit this fact to derive an infinite hierachy of conservation laws. For an analysis of the spectral properties of $L_u$, we refer to Section \ref{sec:dyn_sol} below.

For $u \in H^s_+(\R)$ with  some $s \geq 0$, we formally define the operators $L_u$ and $B_u$ acting on $L^2_+(\R)$ by setting
\be \label{eq:L_B}
\boxed{L_u = D - T_u T_{\bar{u}} \quad \mbox{and} \quad B_u = T_u T_{\pt_x \bar{u}} - T_{\pt_x u} T_{\bar{u}} + \ii (T_u T_{\bar{u}})^2}
\ee
Here $T_b(f) = \Pi_+(b f)$ denotes the {\em Toeplitz operator} on $L^2_+(\R)$ with symbol $b \in L^2(\R) + L^\infty(\R)$. For $u \in H^1_+(\R) \subset  L^\infty(\R)$, we readily check that $T_u$ and $T_{\bar{u}}$ are bounded operators on $L^2_+(\R)$. Thus, for $u \in H^1_+(\R)$, it is straightforward to verify that $L_u$ is semibounded and a self-adjoint operator, i.\,e.,
$$
L_u^* =L_u
$$
with operator domain $\mathrm{dom}(L_u) = H^1_+(\R)$. For $u \in H^2_+(\R)$, we readily check that $B_u = -B_u^*$ is a skew-adjoint and bounded operator on $L^2_+(\R)$. 

\begin{remark}
    In Appendix \ref{app:lax} below, we detail how $L_u$ can be defined via quadratic forms if we only assume that $u \in L^2_+(\R)$, which is a natural class in view of the $L^2$-criticality of (CM-DNLS).
\end{remark}

Next, we see that $L_u$ and $B_u$ form indeed a Lax pair for the Calogero--Moser DNLS.

\begin{lem}[Lax Equation] \label{lem:Lax} If $u \in C([0,T]; H^s_+(\R))$ solves \eqref{eq:NLS} with $s \geq 0$ sufficiently large (e.\,g.~with $s=2$), then it holds
$$
\frac{d}{dt} L_u = [B_u, L_u] .
$$
\end{lem}

\begin{proof}
We divide the proof into the following steps.

\medskip
\textbf{Step 1.} We first calculate the commutators
$$
I := [T_u T_{\pt_x \bar{u}}, D], \quad II := [T_{\pt_x u} T_{\bar{u}}, D], \quad III := \ii [(T_u T_{\bar{u}})^2, D] .
$$
We find
\begin{align*}
I & = T_u [T_{\pt_x \bar{u}}, D] + [T_u, D] T_{\pt_x \bar{u}}   =  T_u T_{\ii \pt_{xx} \bar{u}} + T_{\ii \pt_x u} T_{\pt_x \bar{u}}, \\
II & =  T_{\pt_x u} [T_{\bar{u}}, D] + [T_{\pt_x u}, D] T_{\bar{u}} = T_{\pt_x u} T_{\ii \pt_x \bar{u}} + T_{\ii \pt_{xx} u} T_{\bar{u}}.
\end{align*}
In addition, we see
\begin{align*}
III & = \ii T_u T_{\bar{u}} [T_u T_{\bar{u}}, D] +  \ii [T_u T_{\bar{u}}, D] T_u T_{\bar{u}} \\
& = \ii T_u T_{\bar{u}} \left ( T_u T_{\ii \pt_x \bar{u}} + T_{\ii \pt_x u} T_{\bar{u}} \right ) + \ii \left (  T_u T_{\ii \pt_x \bar{u}} + T_{\ii \pt_x u} T_{\bar{u}} \right ) T_u T_{\bar{u}} .
\end{align*}
As a next step, we consider the terms
$$
IV := [T_u T_{\pt_x \bar{u}}, T_u T_{\bar{u}}], \quad V := [T_{\pt_x u} T_{\bar{u}}, T_u T_{\bar{u}}], \quad VI = \ii [(T_u T_{\bar{u}})^2, T_u T_{\bar{u}}].
$$
We find
$$
IV = T_u T_{\pt_x \bar{u}} T_u T_{\bar{u}} - T_u T_{\bar{u}} T_u T_{\pt_x \bar{u}}, \quad V  = T_{\pt_x u} T_{\bar{u}} T_u T_{\bar{u}} - T_u T_{\bar{u}} T_{\pt_x u} T_{\bar{u}}, \\
$$
and, clearly, we have $VI = 0$. If we combine all commutator terms, we see
\begin{align*}
[B_u, L_u] & =  I - II + III - IV + V \\
& = T_u T_{\ii \pt_xx \bar{u}} - T_{\ii \pt_{xx} u} T_{\bar{u}} -2 T_u T_{\bar{u}} T_{\pt_x u} T_{\bar{u}} - 2 T_u T_{\pt_x \bar{u}} T_u T_{\bar{u}} \\
& = T_u T_{\ii \pt_{xx} \bar{u}} - T_{\ii \pt_{xx} u} T_{\bar{u}} - 2 T_u T_{\pt_x |u|^2} T_{\bar{u}},
\end{align*}
where in the last step we used that $u \in H^1_+(\R)$, which implies  that 
$$
T_{\bar{u}} T_{\pt_x u}  + T_{\pt_x \bar{u}} T_u =  T_{ \bar{u}\pt_x u+\pt_{x} \bar{u} u} = T_{\pt_x |u|^2}
$$ 
holds on $L^2_+(\R)$.

\medskip
\textbf{Step 2.} We now calculate 
\begin{align*}
\frac{d}{dt} L_u &=  -T_{\dot{u}} T_{\bar{u}} - T_u T_{\dot{\bar{u}}} \\
& = -T_{\ii \pt_{xx} u} T_{\bar{u}}  - 2 T_{\Pi_+(\pt_x |u|^2)u} T_{\bar{u}}   - T_u T_{-\ii \pt_{xx} \bar{u}} - 2 T_u T_{ \overline{\Pi_+(\pt_x |u|^2)}\bar{u}}
\end{align*}
In view of the expression for $[B_u, L_u]$ derived in Step 1, it remains to show the identity
\be \label{eq:T_magic}
T_{\Pi_+(\pt_x |u|^2) u} T_{\bar{u}} + T_u T_{\overline{\Pi_+(\pt_x |u|^2)}\bar{u}} = T_u T_{\pt_x |u|^2} T_{\bar{u}}
\ee
Indeed, using that $u \in H^1_+(\R)$, we find 
$$
T_{\Pi_+(\pt_x |u|^2) u} = T_u T_{\Pi_+(\pt_x |u|^2)} \quad \mbox{and} \quad T_{\overline{\Pi_+(\pt_x |u|^2)}\bar{u}} = T_{\overline{\Pi_+(\pt_x |u|^2)}} T_{\bar{u}}.
$$ 
Since $\pt_x |u|^2 = \Pi_+(\pt_x |u|^2) + \overline{\Pi_+(\pt_x |u|^2)}$, we deduce that \eqref{eq:T_magic} holds true. This completes the proof of Lemma \ref{lem:Lax}.
\end{proof}

\begin{remark}
Since we have that $[L_u,L_u^2] =0$, the skew-adjoint operator
$$
\widetilde{B}_u = B_u - \ii L_u^2 
$$
also satisfies the Lax equation $\frac{d}{dt} L_u = [\widetilde{B}_u, L_u]$. A direct calculation shows that
\be \label{def:B_tilde}
\widetilde{B}_u = -\ii D^2 + 2 \ii T_u D T_{\bar{u}}
\ee
Note that $\widetilde{B}_{u}$ is an unbounded skew-adjoint operator on $L^2_+(\R)$ with operator domain $\mathrm{dom}(\widetilde{B}_u) = H^2_+(\R)$. Actually, we first found the operator $\widetilde{B}_u$ in the analysis of the Lax structure of \eqref{eq:NLS}. We also note that the relation between $B_u$ and $\widetilde{B}_u$ is reminiscent to the Lax structure for the Benjamin-Ono equation used in \cite{GeKa-21}.
\end{remark}

As a consequence of the Lax equation, we obtain an infinite hierarchy of conservation of laws for (CM-DNLS) as follows

\begin{lem}[Hierarchy of Conservation Laws]  \label{lem:hierarchy}
Let $u \in C([0,T], H^s_+(\R))$ be a solution of (CM-DNLS) with sufficiently large $s \geq 0$. Then, for every $\lambda \not \in \sigma(L_{u(0)})$, we have the conserved quantity
$$
\mathcal{H}_\lambda(u) := \langle (L_u - \lambda I)^{-1} u, u \rangle.
$$
As a consequence, if $u \in C([0,T]; H^{n/2}_+(\R))$ with some $n \in \N$, the quantities
$$
I_k(u) := \langle  L_u^k u, u \rangle \quad \mbox{with $k=0, \ldots, n$}
$$
are conserved, where $\langle \cdot, \cdot \rangle$ denotes the dual pairing of $H^{-n/2}_+$ and $H^{n/2}_+$.
\end{lem}

\begin{proof}
Let $\lambda \not \in \sigma(L_{u_0})$, which by the Lax equation implies that $\lambda \not \in \sigma(L_u)$ for all $u=u(t)$ with $t \in [0,T]$. A quick calculation reveals that (CM-DNLS) can be written as
$$
\pt_t u = \widetilde{B}_u u
$$
with the operator $\widetilde{B}_u=B_u -i L_u^2$; see also \eqref{def:B_tilde} above. Using that $\frac{d}{dt} L_u =  [\widetilde{B}_u, L_u]$, it is elementary to verify that 
$$
\frac{d}{dt} \mathcal{H}_\lambda(u)=0.
$$
Finally, we note the expansion
$$
\mathcal{H}_\lambda(u) = \sum_{k=0}^\infty \lambda^{-(k+1)} \langle L^k_u u, u \rangle
$$
for all real $\lambda < 0$ sufficiently negative (using that $L_u$ is bounded below) and with $u \in H^\infty_+(\R) = \cap_{s \geq 0} H^s_+(\R)$. Thus we deduce that $I_k(u)=\langle L^k_u u, u \rangle$ are constant in time for solutions in $H^\infty_+(\R)$. The conservation laws $I_0(u), \ldots, I_n(u)$ for solutions $u \in C([0,T]; H^{n/2}_+(\R))$ follow from an approximation argument, which we omit.
\end{proof}

\subsection{Global Existence for sub-critical $L^2$-Mass}
As an application of Lemma \ref{lem:hierarchy}, we deduce the following global-in-time existence result.

\begin{cor}[Global Existence and Bounds for sub-critical $L^2$-Mass] \label{cor:gwp}
Let $u_0 \in H^{s}_+(\R)$ with some integer $s \geq 1$ and suppose that 
$$
M(u_0) < M(\RR) = 2 \pi.
$$ 
Then the corresponding solution $u(t) \in H^{s}_+(\R)$ of (CM-DNLS) exists for all times $t \in \R$ with the a-priori bound
$$
\sup_{t \in \R} \|u(t) \|_{H^{s}} < +\infty.
$$
\end{cor}

\begin{remarks*}
1) The condition $s \geq 1$ arises from the current state of the local well-posedness theory for (CM-DNLS). It is conceivable that, with some great effort though, that this result extends to initial data in $H^0_+(\R)=L^2_+(\R)$ satisfying $\|u_0 \|_{L^2}^2  < 2 \pi$.

2) From Theorem \ref{thm:growth_intro} we deduce that the infinite hierarchy of conservation laws 
$$
I_k(u(t)) = I_k(u_0) \quad \mbox{with $k = 1, 2,\ldots$}
$$
fail in general to produce a-priori bounds on $H^s$-norms for $s> 0$ for solutions with initial data with $L^2$-mass $M(u_0) \geq M(\RR) = 2\pi$. 
\end{remarks*}

\begin{proof}
We first consider the case $s=1$. By the local well-posedness theory in $H^1_+(\R)$, we need to find an a-priori bound on $\sup_{t \in I} \| u(t) \|_{H^1}$, where $I \subset \R$ denotes the maximal time interval of existence. Indeed, we have the conservation laws
$$
I_0(u) = \langle u, u \rangle = \|u \|_{L^2}^2,
$$
$$
I_1(u) = \langle L_u u, u \rangle = \langle Du, u \rangle - \| T_{\bar{u}} u \|_{L^2}^2 .
$$
Now, we use the sharp inequality (see Lemma \ref{lem:GN_basic}):
$$
\| T_{\bar{u}} u \|_{L^2}^2 \leq \frac{1}{2 \pi} \| u \|_{L^2}^2 \langle D u, u \rangle.
$$
Therefore if we assume that $\|u_0 \|_{L^2}^2 < 2 \pi$, we deduce the a-priori bound
$$
\sup_{t \in I} \| u(t) \|_{H^{1/2}} \leq C(u_0).
$$
Next, we use the conservation of energy together with a standard Gagliardo--Nirenberg interpolation and Sobolev inequality:
\begin{align*}
E(u) & = 2 I_2(u) = \langle L^2_u u, u \rangle = \langle Du, Du \rangle  - \frac{1}{2} \langle |u|^2, |D| |u|^2 \rangle + \frac{1}{12} \| u \|_{L^6}^6 \\
& \geq \| \pt_x u \|_{L^2}^2 - C \| u \|_{L^6}^3 \| \pt_x u \|_{L^2} \geq \| \pt_x u \|_{L^2}^2 - C \| u \|_{H^{1/2}}^3 \| \pt_x u \|_{L^2}.
\end{align*}
From the a-priori bound on $\|u(t) \|_{H^{1/2}}$, we readily infer that $\sup_{t \in I} \| \pt_x u(t) \|_{L^2} \leq C(u_0)$. This completes the proof for $s=1$.

The remaining case of integer $s \geq 2$ follows by iteration and using the conserved quantities
$$
I_k(u) = \langle L^k u, u \rangle = \| u \|_{\dot{H}^{k/2}}^2 + \mbox{lower order terms} 
$$
with $k=0, \ldots, 2s$. We omit the details.
\end{proof}

\subsection{Proof of Theorem \ref{thm:gwp} for $M(u_0) < M(\RR)$}
The assertions in Theorem \ref{thm:gwp} in the case $M(u_0) < M(\RR)$ follow directly from Corollary \ref{cor:gwp}. \hfill $\qed$

The critical case $M(u_0) = M(\RR)$ will be discussed in the following section.

\section{Nonexistence of Minimal Mass Blowup}

\label{sec:no_min_blowup}

The goal of this section is to rule out finite-time minimal mass blowup for \eqref{eq:NLS} with finite energy. As a consequence, we obtain that initial data $u(0) \in H^s_+(\R)$ with some $s \geq 1$ with critical $L^2$-mass 
$$
M(u(0)) = M(\RR ) = 2 \pi
$$
will always lead to global-in-time solutions $u \in C(\R; H^s_+(\R))$, completing the proof of Theorem \ref{thm:gwp}. 

Notice that the absence of minimal mass blowup is in striking contrast to focusing $L^2$-critical NLS, where the existence of minimal mass blowup is a direct consequence of applying the pseudo-conformal transform to ground state solitary waves. For \eqref{eq:NLS} on the other hand, we will see below that the mechanism that prevents the existence of minimal mass blowup is due to the slow algebraic decay of ground states $\RR \in H^1_+(\R)$ with $|\RR (x)| \sim \frac{1}{|x|}$ as $|x| \to +\infty$.

We begin with the $L^2$-tightness property for $H^1$-solutions on finite time intervals. 

\begin{lem} \label{lem:L2_tight}
Let $I \subset \R$ be an interval of finite length $|I| < \infty$ and suppose that $u \in C(I, H^1(\R))$ solves \eqref{eq:NLS}. Then the family $\{ u(t) \}_{t \in I}$ is tight in $L^2(\R)$, i.\,e., for every $\eps > 0$ there exists $R=R(\eps) > 0$ such that
$$
\int_{|x| \geq R} |u(t,x)|^2 \, dx \leq \eps \quad \mbox{for all $t \in I$}.
$$
\end{lem}  

\begin{remarks*}
1) It should be noted that no assumption on the size of the solution $u(t)$, i.\,e., we do {\em not} assume that $M(u_0)$ is sufficiently small.  

2) Notice that we do not assume that $u(t)$ belongs to the Hardy space $H^1_+(\R)$, but in fact we allow  for general $H^1$-valued solutions of \eqref{eq:NLS}.
\end{remarks*}

\begin{proof}
The idea is to adapt an elegant argument in \cite{Ba-04} developed for studying minimal mass blowup for $L^2$-critical NLS. Let $\psi \in C^\infty(\R)$ be a smooth real-valued function with bounded derivative $\pt_x \psi \in L^\infty(\R)$. For any $a \in \R$ and $u \in H^1(\R)$, the non-negativity of the energy implies that
$$
E(\mathrm{e}^{i a \psi} u) \geq 0.
$$
Expanding the right-hand side, we find
\begin{align*}
E(\mathrm{e}^{i a \psi} u) & = \frac{1}{2} \int_{\R} |\pt_x(\mathrm{e}^{i a \psi} u) - i \Pi_+(|u|^2) \mathrm{e}^{i a \psi} u|^2 \, dx \\
& = \frac{1}{2} \int_{\R} | i a \pt_x \psi u + \pt_x u - i \Pi_+(|u|^2) u|^2 \, dx \\
& = \frac{a^2}{2} \int_{\R} |\pt_x \psi|^2 |\pt_x u|^2 \, dx + a \mathrm{Re} \, \langle i \pt_x \psi u, \pt_x u - i \Pi_+(|u|^2) u \rangle + E(u) .
\end{align*}
Now we observe that
$$
\mathrm{Re} \, \langle i \pt_x \psi u, \pt_x u \rangle = \int_\R \pt_x \psi \cdot \mathrm{Im}(\ov{u} \pt_x u) \, dx,
$$
$$
 \mathrm{Re} \langle i \pt_x \psi u,  - i \Pi_+(|u|^2) u \rangle = -\frac{1}{2} \mathrm{Re} \langle \pt_x \psi u, (1+i \Hil)(|u|^2) u \rangle = -\frac{1}{2} \int_\R \pt_x \psi |u|^4 \, dx,
 $$
 using that $\Pi_+=\frac{1}{2} (1+i H)$ and $\mathrm{Re} \langle u, iH(|u|^2) u \rangle = 0$ since $H(|u|^2)$ is real-valued. Therefore, the quadratic expansion in $a$ together with $E(\mathrm{e}^{i a \psi} u) \geq 0$ implies that
 \be \label{ineq:banica}
 \left | \int_{\R} \pt_x \psi \left ( \mathrm{Im}(\ov{u} \pt_x u) - \frac{1}{2} |u|^4 \right )  dx \right | \leq \sqrt{2 E_0} \left ( \int_{\R} |\pt_x \psi|^2 |\pt_x u|^2 \,d x \right )^{1/2},
 \ee
with the energy $E_0 = E(u) \geq 0$. 
 
 Next, we apply \eqref{ineq:banica} to obtain the claimed $L^2$-tightness bound. From \eqref{eq:NLS} we deduce
\be \label{eq:con_rho}
 \pt_t |u|^2 =-2 \pt_x\left (  \mathrm{Im} (\ov{u} \pt_x u) - \frac{1}{2}  |u|^4 \right )
\ee
in view of $\mathrm{Re} ( \ov{u} (2i D_+|u|^2) u) =  \mathrm{Re}( (\pt_x |u|^2) |u|^2) = \frac{1}{2} \pt_x |u|^4$. Now let $\chi$ be a smooth nonnegative function such that $\chi(x) \equiv 0$ for $|x| \leq 1/2$ and $\chi(x) \equiv 1$ for $|x| \geq 1$. For $R > 0$, we set $\chi_R(x) = \chi(x/R)$. Integrating by parts and using \eqref{eq:con_rho} and \eqref{ineq:banica}, we infer that
\begin{align} \label{ineq:banica2}
\left | \frac{d}{dt} \int_{\R} \chi_R |u|^2 \, dx \right | & = 2 \left  | \int_\R \pt_x \chi_R \left ( \mathrm{Im}(\ov{u} \pt_x u) - \frac{1}{2} |u|^4 \right )  \, dx \right | \\
& \lesssim \sqrt{E_0} \left ( \int_{\R} |\pt_x \chi_R|^2 |u|^2 \,d x \right )^{1/2} \lesssim \frac{\sqrt{E_0 M_0}}{R} 
\end{align}
with the conserved $L^2$-mass $M_0 = \| u \|_{L^2}^2$. If we integrate this bound over the finite time interval $I \subset \R$ with some $t_0 \in I$ fixed, we finally obtain
$$
\int_{\R} \chi_R(x) |u(t,x)|^2 \, dx \leq \int_\R \chi_R(x) |u(t_0,x)|^2 \, dx + \frac{C |I|}{R}, 
$$
for all $t \in I$. This readily implies the claimed tightness bound.
\end{proof}

\begin{thm}[No Minimal Mass Blowup] \label{thm:min_mass}
Let $I \subset \R$ with $0 \in I$ and finite length $|I| < \infty$. Suppose $u \in C(I; H^1(\R))$ solves \eqref{eq:NLS} with $M(u_0) = M(\RR) = 2\pi$. Then it holds
$$
\sup_{t \in I} \| u(t) \|_{H^1} < +\infty.
$$
\end{thm}

\begin{remark*}
1) Note again that we allow for general $H^1$-valued solutions $u(t)$. 
\end{remark*}

\begin{proof} 
We argue by contradiction. Without loss of generality, we can assume that $I=[0,T)$ with some finite time $T \in (0, +\infty)$ and let $u \in C([0,T); H^1(\R))$ satisfy
$$
\lim_{t \to T^-} \| u(t) \|_{H^1} = +\infty.
$$

\medskip
\textbf{Step 1.} We first show that $u(t)$ must have finite variance, i.\,e., we have 
\be \label{ineq:fin_var}
\int_{\R} |x|^2 |u(t,x)|^2 \, dx < +\infty \quad \mbox{for $t \in [0,T)$.}
\ee
Here we adapt a strategy developed for classifying minimal-mass finite-time blowup solutions for $L^2$-critical NLS; originally due to Merle in \cite{Me-93}. To prove the claim \eqref{ineq:fin_var}, we follow the arguments laid out in \cite{HmKe-05, Ba-04}.

Let $t_n \to T^-$ be a sequence of times. We define 
$$\eps_n :=\frac{1}{\Vert \partial_xu(t_n)\Vert_{L^2}}\ ,\ v_n(x):=\eps_n^{1/2}u(t_n,\eps_nx)\ .$$
Applying the Minimal Mass Bubble Lemma \ref{lem:MMB} from Appendix \ref{varE} to $v_n$, after passing to a subsequence if necessary, there exist sequences  $x_n \in \R$ and $\lambda_n > 0$ such that $\lambda_n \to 0$ and 
$$
\lambda_n^{1/2} u(t_n, \lambda_n (x+ x_n)) \to \mathrm{e}^{i \theta } \RR (x) \quad \mbox{strongly in $L^2(\R)$}
$$
for some $\theta \in [0,2\pi[$, with the ground state $\RR \in H^1_+(\R)$ minimizing the energy functional $E(u)$ on $H^1(\R)$. Thus we obtain
$$
|u(t_n, x)|^2 \, dx - \| \RR \|_{L^2}^2 \delta_{x=\lambda_n x_n} \weakto 0
$$
in the weak sense of measures. By the $L^2$-tightness property in Lemma \ref{lem:L2_tight}, we easily deduce that $\lambda _n|x_n| \leq C$ with some constant $C>0$. From this fact (and passing to a subsequence) and by translational invariance we can henceforth assume that $\lambda_n x_n \to 0$ holds. 

Next, let $\psi \in C^\infty_0(\R)$ be a non-negative function such that $\psi(x) \equiv |x|^2$ for $|x| < 1$ and $|\pt_x \psi(x)|^2 \leq C \psi(x)$ with some constant $C>0$. For $R> 0$, we set $\psi_R(x) = R^2 \psi(x/R)$ and we define
$$
g_R(t) = \int_\R \psi_R(x) |u(t,x)|^2 \, dx.
$$
In analogy to \eqref{ineq:banica2} based on \eqref{ineq:banica} we find
\be
\left |\frac{d}{dt} {g}_R(t) \right | \lesssim \int_{\R} |\pt_x \psi_R|^2  |u|^2 \,dx \lesssim \sqrt{g_R(t)}
\ee 
where the last step used that $|\pt_x \psi_R|^2 \lesssim \psi_R$ by construction. By integrating this on $[t, t_n]$ and using that $g_R(t_n) \to 0$ as $n \to \infty$, we deduce
$$
g_R(t) = \int_\R \psi_R(x) |u(t,x)|^2 \,dx \lesssim (T-t)^2 \quad \mbox{for $t \in [0,T)$}.
$$
Passing to the limit $R \to +\infty$, this yields
\be
\int_\R |x|^2 |u(t,x)|^2 \, dx \lesssim (T-t)^2 \quad \mbox{for $t \in [0,T)$}.
\ee 
In particular, this implies that \eqref{ineq:fin_var} holds, showing that $u(t)$ has finite variance.

\medskip
\textbf{Step 2.}  By Step 1, we have $u_0=u(0) \in \Sigma = H^1(\R) \cap L^2(\R; |x|^2 dx)$ and thus we can apply the pseudo-conformal identity (see Lemma \ref{lem:pseudo}) to conclude that  
\be
8 t^2 E(\mathrm{e}^{i |x|^2/4t} u_0) = \int_{\R} |x|^2 |u(t,x)|^2 \, dx \lesssim (T-t)^2 .
\ee
If we pass to the limit $t \to T^-$, we obtain
\be
E(\mathrm{e}^{i |x|^2/4T} u_0) = 0.
\ee
Since $\|\mathrm{e}^{i |x|^2/4T} u_0 \|_{L^2}^2 = 2 \pi$, the uniqueness result in Lemma \ref{lem:R_uniq} implies that 
\be
\mathrm{e}^{i |x|^2/4T} u_0(x)= \RR(x)
\ee
up to translation, phase, and scaling. However, the fact that $\mathrm{e}^{i |x|^2/4T} u_0 \in \Sigma$ contradicts that $\RR(x)$ has infinite variance, i.\,e.,
$$
\int_\R |x|^2 |\RR(x)|^2 \, dx =  +\infty.
$$
This contradiction shows that any  $u(t) \in H^1(\R)$ solving \eqref{eq:NLS} with $M(u(0)) = M(\RR) = 2 \pi$ cannot blowup in finite time. This completes the proof of Theorem \ref{thm:min_mass}. Therefore, the proof of Theorem \ref{thm:gwp} is also complete.
\end{proof}

\section{Ground States and Traveling Solitary Waves}

\label{sec:gs}

In this section, we show that all ground states (minimizers) for the energy functional $E(u)$ for \eqref{eq:NLS} are given by the rational function
\be  \label{def:R}
\RR(x) = \frac{\sqrt{2}}{x+i} \in H^1_+(\R)
\ee
modulo translation, phase, and scaling. As as a second main result of this section, we prove that all traveling solitary waves in $H^1_+(\R)$ for \eqref{eq:NLS} are given by Galilean boosts (with positive velocity) of $\RR$ up to the symmetries just mentioned.

We start with the following key result, which shows uniqueness of non-trivial solutions of the (first-order) Euler-Lagrange equation for minimizers of $E(u)$.

\begin{lem}[Uniqueness of Ground States] \label{lem:R_uniq}
Suppose that $u \in H^1(\R)$ with $u \not \equiv 0$ solves 
$$
D u - \Pi_+(|u|^2) u = 0.
$$ 
Then it holds
$$
u(x) = \mathrm{e}^{i \theta} \lambda^{1/2} \RR(\lambda x+y)
$$
with some $\theta \in [0, 2\pi)$, $\lambda > 0$, $y  \in \R$, and $\RR \in H^1_+(\R)$ is  given by \eqref{def:R}.

As a consequence, all minimizers $u \in H^1(\R) \setminus \{ 0 \}$ for $E(u)$ are of the form $u=\RR$ modulo phase, translation, and scaling.
\end{lem}

\begin{remarks*}
1) Note we that only assume that $u \in H^1(\R)$ and  we obtain {\em a posteriori} that $u \in H^1_+(\R)$ due to its explicit form. 

2) As an interesting aside, we remark that the equation for $u \in H^1(\R)$ above can be recast into the {\em nonlocal Liouville equation} on the real line: 
\be \label{eq:liouville}	
|D| w = \mathrm{e}^w \quad \mbox{in $\R$}.
\ee 
To prove this claim (neglecting any technicalities of function spaces for simplicity), we first apply the gauge transform introducing the function $v = \mathrm{e}^{-i/2 \int_{-\infty}^x |u|^2} u$, which leads to the equation for $v$ given by
\be \label{eq:v}
\pt_x v + \frac{1}{2} \Hil(|v|^2) v = 0 \quad \mbox{in $\R$},
\ee
where we recall that $\Hil$ denotes the Hilbert transform. It is easy to see non-trivial solutions $v$ are (up to a complex phase factor) strictly positive $v>0$. Finally, if we define $w : \R \to \R$ by setting $w= \log(v^2)$, we readily check that $w$ solves \eqref{eq:liouville}. In \cite{ChYa-97,Xu-05,LiMa-17} it is proven that all solutions $w \in L^1(\R; \frac{dx}{1+x^2})$ of \eqref{eq:liouville} are explicitly given by
\be
w(x) = \log \left ( \frac{2 \lambda}{1+\lambda^2(x-y)} \right )
\ee
with some constants $\lambda > 0$ and $y \in \R$. From this uniqueness result for $w$, we could obtain the result of Lemma \ref{lem:R_uniq}. However, we will below give a self-contained (and short) uniqueness proof based on Hardy space arguments and complex ODEs. This method  also provides (yet) another proof of the stated uniqueness result for $w$ solving \eqref{eq:liouville}.  See also \cite{AhLe-22} for a recent uniqueness result for solutions $w$ of the nonlocal Liouville equation $|D| w = Ke^w$ in $\R$ with the prescribed function $K : \R \to \R$.
\end{remarks*}

\begin{proof}[Proof of Lemma \ref{lem:R_uniq}]
We introduce the function
$$
w := \Pi_+(|u|^2) \in H^1_+(\R),
$$
so that $u' = \ii uw$ and $|u|^2 = w + \ov{w}$. We obtain the complex ordinary differential equation
$$
w' = \Pi_+ (u' \ov{u} + \ov{u}' u)  = \Pi_+( \ii w(w+ \ov{w}) - \ii \ov{w}(\ov{w}+w)) = \ii w^2,
$$
using that $\Pi_+(\ov{w}^2)=0$ for $w \in H^1_+(\R)$. Since $w \not \equiv 0$, we deduce
$$
w(x) = \frac{\ii}{x -z}
$$
for some constant $z \in \C$ with $\mathrm{Im} \, z <0$ because $w \in L^2_+(\R)$. Consequently,
$$
u' = \ii u w = -\frac{u}{x-z} .
$$
Thus we have $(x-z) u(x) = c$ with some constant $c \in \C$. To determine $c$, we notice
$$
\frac{\ii}{x-z} = \Pi_+(|u|^2) =  \Pi_+\left ( \frac{|c|^2}{|x-z|^2} \right ) = \frac{|c|^2}{(\ov{z}-z)(x-z)} . 
$$
This implies that $|c|^2 = - 2 \mathrm{Im} \, z = 2 |\mathrm{Im} \, z|$. In summary, we have found that
$$
u(x) = \mathrm{e}^{\ii \theta} \frac{\sqrt{2 |\mathrm{Im} \, z|}}{x-z} 
$$
with some $\theta \in [0, 2\pi)$ and $z \in \C_-$. From this fact we readily deduce that $u(x) = \mathrm{e}^{i \theta} \lambda^{1/2}\RR(\lambda x + y)$ with $\lambda = -(\mathrm{Im} z)^{-1}>0$ and $y=-\mathrm{Re} \, z$. 
\end{proof}

Our next goal is to classify all traveling solitary wave solutions for \eqref{eq:NLS}.  By a \textbf{traveling solitary waves} (with finite energy), we mean solutions of the form
\be
u(t,x) = \mathrm{e}^{i \omega t} \RR_{v,\omega}(x-v t)
\ee
where $\omega \in \R$ is a frequency parameter and $v \in \R$ denotes the velocity. Here the non-trivial profile $\RR_{v,\omega} \in H^1(\R)$ is allowed to depend on $\omega$ and $v$. Note that a-priori we allow also for $\RR_{v,\omega}$ in $H^1(\R)$ and not just restricted to $H^1_+(\R)$. We have the following complete classification result, which shows that all traveling solitary waves for \eqref{eq:NLS} are generated by Galilean boosts, translations, scaling and phase transformations of the ground state $\RR(x)$. 

\begin{prop} \label{prop:travel}
Let $v, \omega \in \R$ and suppose $u(t,x)$ is a traveling solitary wave for \eqref{eq:NLS} with profile $\RR_{v,\omega} \in H^1(\R) \setminus \{ 0 \}$. Then it holds
$$
\omega = -\frac{v^2}{4} \quad \mbox{and} \quad \RR_{v,\omega}(x) = \mathrm{e}^{\frac{\ii}{2} v x} \mathrm{e}^{i \theta} \lambda^{1/2} \RR(\lambda x + y)
$$
with some $\theta \in [0, 2\pi), \lambda > 0$, and $y \in \R$. Moreover, we have $\RR_{v, \omega} \in H^1_+(\R)$ if and only if $v \geq 0$ holds.
\end{prop}

\begin{proof}
We divide the proof into the following steps.

\medskip
\textbf{Step 1.} We easily check that $\RR_{v,\omega} \in H^1(\R)$ must solve
\be
-\pt_{xx} \RR_{v, \omega} + \ii v \pt_x \RR_{v,\omega} - 2D_+(|\RR_{v,\omega}|^2) \RR_{v,\omega} = \omega \RR_{v,\omega}.
\ee
The term $\ii v \pt_x \RR_{v, \omega}$ can be removed by a Galilean boost transform. That is, we write
\be
\RR_{v,\omega}(x) = \mathrm{e}^{\frac{\ii}{2} vx} \RR_{\tilde{\omega}}(x).
\ee
An elementary calculation yields that $\RR_{\tilde{\omega}} \in H^1(\R)$ satisfies
\be \label{eq:R_w}
-\pt_{xx} \RR_{\tilde{\omega}} - 2 D_+(|\RR_{\tilde{\omega}}|^2) \RR_{\tilde{\omega}} = \tilde{\omega} \RR_{\tilde{\omega}} \quad \mbox{with} \quad \tilde{\omega} = \omega + \frac{v^2}{4}.
\ee

\medskip
\textbf{Step 2.} We claim that every solution $\RR_{\tilde{\omega}} \in H^1(\R)$ of \eqref{eq:R_w} has zero energy:
\be \label{eq:E_poho}
E(\RR_{\tilde{\omega}}) = 0.
\ee
Since $\RR_{\tilde{\omega}} \not \equiv 0$, we see from Lemma \ref{lem:R_uniq} that 
$$
\RR_{\tilde{\omega}} = \RR
$$ 
modulo symmetries. As a consequence, we obtain $\tilde{\omega} = 0$ and therefore $\omega = -\frac{v^2}{4}$ as claimed.

The proof of \eqref{eq:E_poho} follows from applying the gauge transform together with a Pohozaev-type argument. Since $D_+(|\RRw|^2) \in L^2(\R)$, we notice that equation \eqref{eq:R_w} tells us that $\RRw \in H^2(\R)$ holds. Next, we apply the gauge transform $\Phi$ discussed in Appendix \ref{app:misc} to $\RRw$. That is, we set
\be \label{eq:gauge2}
S(x) := \Phi(\RRw)(x) = \mathrm{e}^{-\frac{i}{2} \int_{-\infty}^x |\RRw(y)|^2 \, dy} \RRw(x).
\ee 
We directly check that $S \in H^2(\R)$ and $|S|^2 = |\RRw|^2$. Using that $|D| = \Hil \pt_x$ and $\Pi_+ = \frac{1}{2}(1+\ii \Hil)$, a calculation yields that
\be \label{eq:S_om}
-\pt_{xx} S - (|D| |S|^2) S + \frac{1}{4} |S|^4 S = \tilde{\omega} S.
\ee
If we integrate this equation against $\ov{S}$, we directly obtain
\be \label{eq:poho1}
\int_{\R}  |\pt_x S|^2 - \int_\R |S|^2 (|D| |S|^2) + \frac{1}{4}  \int_\R |S|^6  = \tilde{\omega} \int_{\R} |S|^2.
\ee
Next, we  integrate \eqref{eq:S_om} against $x \pt_x \ov{S}$ over the compact interval $[-R,R]$. By taking the real part and taking the limit $R \to +\infty$, we find the identity
\be \label{eq:poho2}
\frac{1}{2} \int_{\R} |\pt_x S|^2 - \frac{1}{24} \int_{\R} |S|^6  = -\frac{\tilde{\omega}}{2} \int_{\R} |S|^2.
\ee
For details of this step, we refer to Appendix \ref{app:misc}. The combination of \eqref{eq:poho1} and \eqref{eq:poho2} yields
\be
 \frac{1}{2} \int_\R |\pt_x S|^2 - \frac{1}{4} \int_{\R} |S|^2 (|D| |S|^2) + \frac{1}{24} \int_{\R} |S|^6 = 0.
\ee
Recall that the right side can be written as a complete square (see Appendix \ref{app:misc})
$$
\frac{1}{2} \int_\R |\pt_x S + \frac{1}{2} \Hil(|S|^2) S|^2 \, dx = \frac{1}{2} \int_\R |\pt_x \RRw - \frac{i}{2} \Pi_+(|\RRw|^2) \RRw |^2 \,dx,
$$
where the last step follows by using \eqref{eq:gauge2}. Thus we deduce that \eqref{eq:E_poho} holds, which completes the proof.
\end{proof}

\subsection{Proof of Theorem \ref{thm:trav_wave}}
This claim immediately follows from Proposition \ref{prop:travel} by taking $\eta = v/2$. \hfill $\qed$

\section{Analysis of the Lax Operator}

\label{sec:lax}

In this section, we will further study the Lax operator 
$$
L_u = D - T_u T_{\bar{u}}
$$
introduced in Section \ref{sec:WP} above, where $T_b(f) = \Pi_+ (bf)$ denotes the Toeplitz operator on $L^2_+(\R)$ with symbol $b$. In particular, we will derive important commutator identities and prove simplicity of eigenvalues of $L_u$ for general potentials $u \in H^1_+(\R)$ along with optimal bounds on the number of eigenvalues of $L_u$. Moreover, we will define the notion of \textbf{multi-soliton potentials} below.

\subsection{Simplicity of Eigenvalues and Sharp Bounds for $L_u$}
Given $u\in H^1_+(\R )$\footnote{Recall that the regularity assumption  can be relaxed to $u \in L^2_+(\R)$ via an approach using quadratic forms; see Appendix \ref{app:lax} below.}, we recall that the  operator $L_u$ defined above is an unbounded self-adjoint operator on the Hardy space $L^2_+(\R )$, with 
the operator domain $H^1_+(\R )$. First we show simplicity of eigenvalues for $L_u$ together with a bound on the number of eigenvalues.

\begin{prop}\label{prop:L_basic}
If $L_u\psi =\lambda \psi$, then 
\be \label{identity}
|\inner u \psi |^2=2\pi \norm {\psi}_{L^2}^2\ .
\ee
In particular, every eigenvalue of $L_u$ is simple, and the number $N$ of eigenvalues of $L_u$ is finite with 
$$N\leq \frac{\norm{u}_{L^2}^2}{2\pi}\ .$$
In particular, the operator $L_u$ has no point spectrum if $\| u \|_{L^2}^2 < 2 \pi$.
\end{prop}

\begin{remark*}
In \cite{Wu-16}, an identity reminiscent to \eqref{identity} was derived for the Lax operator of the Benjamin-Ono equation on the real line.  
\end{remark*}

\begin{proof}
The key is to introduce the {\em Lax--Beurling semigroup of contractions} defined on $L^2_+(\R )$ as
$$S(\eta )f (x)={\rm e}^{i\eta x}f(x)\ ,\ \eta \ge 0\ .$$
and the corresponding {\em adjoint semigroup} defined as
$$S(\eta )^*f=\Pi_+\left ({\rm e}^{-ix\eta }f\right )\ ,\ \eta \ge 0\ .$$
Clearly, $S(\eta )$ and $S(\eta )^*$ act on the domain $H^1_+(\R)$ of $L_u$. Also notice that, via the Fourier transform, we have
$$\widehat{S(\eta)^*f}(\xi )=\hat f(\xi +\eta )\ ,\ \xi\ge 0, \eta \ge 0\ .$$
To complete the proof of Proposition \ref{prop:L_basic}, we need the following identity.

\begin{lem}\label{lem:firstbracket}
For every $f\in H^1_+(\R)$,  we have the following limit in the $L^2$-norm,
$$\lim_{\eta \to 0}\left [\frac{S(\eta )^*}{\eta }, L_u \right ]f=f-\frac{1}{2\pi}\inner f u u\ .$$
\end{lem}
 
\begin{proof}[Proof of Lemma \ref{lem:firstbracket}]
Notice that, for $\xi >0$ , 
$$\widehat{L_uf}(\xi )=\xi \hat f(\xi)-\frac{1}{4\pi^2}\int_0^\xi \hat u(\xi-\zeta)\left [\int_0^\infty \overline{\hat u (\tau)}\hat f(\zeta +\tau)\, d\tau\right ]\, d\zeta \ .$$
Set
$$A(\eta ):=\left [\frac{S(\eta )^*}{\eta }, L_u \right ]\ ,\ \xi \ge 0\ .$$
Then
$$\widehat{A(\eta)f}(\xi )=\hat f(\xi +\eta )
-\frac{1}{4\pi ^2\eta}\int_0^\eta \hat u(\xi+\eta -\zeta)\left [\int_0^\infty \overline{\hat u (\tau)}\hat f(\zeta +\tau)\, d\tau\right ]\, d\zeta \ .
$$
Passing to the limit in $L^2$, we get
$$\lim_{\eta \to 0}\widehat{A(\eta)f}(\xi )=\hat f(\xi)-\frac{1}{4\pi^2}\hat u(\xi )\int_0^\infty \overline{\hat u (\tau)}\hat f(\tau)\, d\tau\ ,$$
which yields Lemma \ref{lem:firstbracket}. 
\end{proof}
For future reference, we state the following lemma, which can be proved similarly to Lemma \ref{lem:firstbracket}.
\begin{lem}\label{lem:firstbracketbis}
Let $a, b\in L^2_+(\R )$. For every $f\in L^2_+(\R)$,  we have the following limit in the $L^2$-norm,
$$\lim_{\eta \to 0}\left [\frac{S(\eta )^*}{\eta },  T_aT_{\overline b} \right ]f=\frac{1}{2\pi}\inner f b a\ .$$
\end{lem}

Let us come back to the proof Proposition \ref{prop:L_basic}. If we apply Lemma \ref{lem:firstbracket} to $f=\psi $, we obtain
$$\lim_{\eta \to 0}\frac{(\lambda I-L_u)S(\eta )^*\psi }{\eta}=\psi-\frac{1}{2\pi}\inner \psi u u\ .$$
Taking inner product of both sides with $\psi $, we obtain
$$0=\| \psi \|_{L^2}^2-\frac{1}{2\pi}|\inner \psi u |^2\ $$
which is \eqref{identity}. Now the next statements of Proposition \ref{prop:L_basic} follow easily. Indeed,
identity \eqref{identity} implies that the kernel of the linear form $\inner . u$ on any eigenspace is reduced to $\{ 0 \}$, so  that this eigenspace is one dimensional. Finally, if $\psi_1,\dots ,\psi_N$ is an orthonormal system of eigenvectors, we have
$$\| u  \|_{L^2}^2\geq \sum_{j=1}^n |\inner u {\psi_j} |^2=2\pi N\ .$$
This proof of Proposition \ref{prop:L_basic} is now complete.
\end{proof}


\subsection{Multi-Soliton Potentials}
Next, we exhibit a class of potentials $u \in H^1_+(\R)$ which are given by rational functions and which optimize the general bound for the number $N$ of eigenvalues of $L_u$ found in Proposition \ref{prop:L_basic} above. To this end, we use $\sigma_{pp}(L_u)$ to denote the  pure point spectrum of $L_u$ and correspondingly we define
$$
\mathcal{E}_{pp}(u) ={\rm span}  \{ \psi \in \mathrm{ker} (L_u - \lambda I) :  \lambda \in \sigma_{pp}(L_u) \}
$$
to be the space spanned by the eigenfunctions of $L_u$. We have the following spectral characterization result.

\begin{prop}\label{prop:spectralNsoliton}
Let $u \in H^1_+(\R)$ be given and let $N \geq 1$ be an integer. Then the following properties are equivalent and preserved by the flow of \eqref{eq:NLS}.
\begin{enumerate}
\item[(i)] The space $\mathcal{E}_{pp}(u)$ has dimension $N$ and is invariant under the adjoint semigroup $\{ S(\eta)^*\}_{\eta \geq 0}$ and it holds that $u \in \mathcal{E}_{pp}(u)$.
\item[(ii)] There exist a polynomial $Q \in \C[x]$ of degree $N$, with all its zeros in $\C_-$, and a polynomial $P \in \C[x]$ of degree at most $N-1$ such that
$$
u(x) = \frac{P(x)}{Q(x)} \quad \mbox{and} \quad P \overline{P} = i (Q' \overline{Q} - \overline{Q}' Q ).
$$
\end{enumerate}
\end{prop}

Before we prove this result, let us make some general comments as follows. We refer to the rational functions $u \in H^1_+(\R)$ above as a \textbf{multi-soliton potential} or \textbf{$N$-soliton potential}. In Section \ref{sec:dyn_sol} below, we will study the time evolution of multi-soliton potentials, which correspond to \textbf{multi-solitons}. In view of Proposition \ref{prop:L_basic}, we see that multi-solitons must have a quantized $L^2$-mass according to
$$
\| u \|_{L^2}^2 = 2\pi N
$$
with some integer $N \geq 1$. As a generic example for a multi-soliton potential, we can take the polynomial $Q$ to have simple zeros, i.\,e., 
$$
Q(x)=\prod_{j=1}^N (x-z_j),
$$ 
where $z_1, \ldots, z_N \in \C_-$ are pairwise distinct. Then the condition in (ii) above  yields that 
$$
u(x)=\sum_{j=1}^N\frac{a_j}{x-z_j} \quad \mbox{with} \quad \sum_{j=1}^N \frac{a_j\overline a_k}{z_j-\overline z_k}=\ii \quad \mbox{with $k=1,\dots ,N$} .
$$

Let us also remark that, for a given polynomial $Q \in \C[x]$ of degree $N \geq 1$, there exist only a finite number of polynomials $P \in \C[x]$ satisfying the constraint in (ii) up to a constant complex phase. We can find them as follows. Consider the polynomial $F:=i(Q'\overline Q-\overline Q'Q) \in \C[x]$ of degree $2N-2$. Suppose that 
$$Q(x)=\prod_{j=1}^N (x-z_j)$$
with  zeros $z_1,\dots ,z_N \in \C_-$, which are not necessarily distinct. For $x \in \R$, we get 
$$\frac{F(x)}{|Q(x)|^2}=\sum_{j=1}^N \frac{-2{\rm Im}z_j}{|x-z_j|^2}>0\ .$$
In particular, the zeroes of $F$ come in pairs as $(\alpha_j,\overline \alpha_j)$, $j=1,\dots ,N-1$, with $\alpha _j \notin \R$. Thus we can write 
$$P(x)=c\prod_{j=1}^{N-1} (x-\alpha_j),$$
where the constant $c \in \C$ is adjusted so that $P\overline P=F$ holds. Of course, exchanging one $\alpha_j $ with $\overline \alpha_j$ leads to a different $P$ and hence to a different function $u$.

\begin{proof}[Proof of Proposition \ref{prop:spectralNsoliton}]
We first proof the equivalence of statements (i) and (ii). Finally, we address the preservation by the flow of (CM-DNLS).

\medskip
\textbf{Step 1: (ii) $\Rightarrow$ (i)}. 
For $u(x) = \frac{P(x)}{Q(x)} \in H^1_+(\R)$ as in (ii), we define 
$$\theta (x)=\frac{\overline Q(x)}{Q(x)} \quad \mbox{and} \quad K_\theta := \frac{\C_{N-1}[x]}{Q(x)}\ ,$$
so that we have the orthogonal decomposition
$$L^2_+(\R) = K_\theta \oplus \theta L^2_+(\R) \ .$$
We claim that $K_\theta$ is an invariant subspace of $L_u$. Indeed, if $f=A/Q\in K_\theta $, then
\begin{eqnarray*}
L_u(f)&=&\frac{-iA'}{Q}+\frac{iQ'A}{Q^2}-\frac{P}{Q}\Pi_+\left ( \frac{\overline P}{\overline Q} \frac{A}{Q}  \right )\\
&=&\frac{-iA'}{Q}+\Pi_+\left [\frac{P\overline P}{Q\overline Q}\frac AQ+i\frac{\overline Q'}{\overline Q}\frac AQ\right ]-\frac{P}{Q}\Pi_+\left ( \frac{\overline P}{\overline Q} \frac{A}{Q}  \right )\\
&=&\frac{-iA'}{Q}+i\Pi_+\left (  \frac{\overline Q'}{\overline Q} \frac{A}{Q}\right ) +\Pi_+\left (\frac PQ \Pi_-\left (\frac{\overline PA}{\overline Q Q}\right )  \right )\ .
\end{eqnarray*}
Observe that, if $R\in L^2(\R )$ is  a rational function, and if the denominator of $R$ reads $Q_+Q_-$, where $Q_\pm $ is a polynomial 
with zeroes in $\C_\pm $, then $\Pi_+(R)$  is of the form $P/Q_-$, where $P$ is a polynomial of degree less than the degree of $Q_-$. From this observation and the above identity, we conclude that  $L_u(f)\in K_\theta $. Therefore $L_u$ is a self-adjoint endomorphism on the finite dimensional space $K_\theta $, which has dimension $N$. This implies that $L_u$ has at least $N$ eigenvalues. Since $u \in K_\theta$, it follows that $u$ is a linear combination of an orthonormal basis of eigenfunctions $\psi_1, \ldots, \psi_{N} \in K_\theta$. From Proposition \ref{prop:L_basic} we conclude that
 $$\norm{u}_{L^2}^2= \sum_{k=1}^{N} |\langle u, \psi_k \rangle |^2 = 2\pi N.$$
By invoking Proposition \ref{prop:L_basic} again, we deduce that $L_u$ has exactly $N$ eigenvalues. Thus we have shown $\mathcal{E}_{pp}(u) = K_\theta$ and therefore we conclude $\dim \mathcal{E}_{pp}(u) = N$ as well as $u \in \mathcal{E}_{pp}(u)$. Since the semigroup $\{ S(\eta) \}_{\eta \geq 0}$ leaves the space $\theta L_+^2(\R)$ invariant, we obtain that its adjoint semigroup $\{ S(\eta)^* \}_{\eta \geq 0}$ leaves its orthogonal complement $\mathcal{E}_{pp}(u) = K_\theta$ invariant. 

\medskip
\textbf{Step 2: (i) $\Rightarrow$ (ii)}. Suppose that $\mathcal{E}_{pp}(u)$ has dimension $N \geq 1$ and is invariant under the adjoint semigroup $\{ S(\eta)^* \}_{\eta \geq 0}$. Thus the orthogonal complement $(\mathcal{E}_{pp}(u))^\perp$ is invariant preserved by the action of the semigroup $S(\eta)$ with $\eta \geq 0$. By the Lax--Beurling theorem \cite{Lax-59}, we conclude
$$
(\mathcal{E}_{pp}(u))^\perp = \theta L^2_+(\R)
$$
with some inner function $\theta$ defined on the upper complex halfplane $\C_+$. Furthermore, since $\mathcal{E}_{pp}(u)$ is $N$--dimensional, one can choose $\theta $ of the form
$$\theta (x)=\frac{\overline Q(x)}{Q(x)}\ ,\ Q(x)=\prod_{j=1}^N (x-z_j)\ ,\ \mathrm{Im}z_j<0\ .$$
Consequently,
$$\mathcal{E}_{pp}(u)=(\theta L^2_+(\R))^\perp=K_\theta =\frac{\C_{N-1}[x]}{Q(x)}\ .$$
Since $u\in \mathcal{E}_{pp}(u)$, there exists $P\in \C_{N-1}[x]$ such that $u=P/Q$. Since $L_u$ is self-adjoint, we have
$$L_u[(\mathcal{E}_{pp}(u))^\perp \cap H^1_+(\R )]\subset (\mathcal{E}_{pp}(u))^\perp \ ,$$
alternatively $L_u(\theta H^1_+(\R ))\subset \theta L^2_+(\R )$. Let $h\in H^1_+(\R )$. We have
$$L_u(\theta h)=\theta Dh+(D\theta )h -u\Pi_+(\overline u\theta h)=\theta Dh+(D\theta )h -|u|^2\theta h\ ,$$
because 
$$\overline u\theta =\frac{\overline P}{\overline Q }\frac{\overline Q}{Q}=\frac{\overline P}{Q}\in L^2_+(\R )\ .$$
We infer $(D\theta )h -\theta |u|^2h \in \theta L^2_+(\R )$ for every $h\in H^1_+(\R )$, or
$$\frac{D\theta }{\theta }-|u|^2\in L^2_+(\R )\ .$$
Notice that, for every $x\in \R$, 
$$\frac{D\theta }{\theta }=\sum_{j=1}^N \frac{2\mathrm{Im}z_j}{|x-z_j|^2}$$
hence $D\theta /\theta -|u|^2$ is real valued, therefore it belongs to $L^2_+(\R )\cap \overline {L^2_+(\R )}=\{ 0\}$. Reformulating this identity in terms of $P$ and $Q$, we obtain $P\overline P=i(Q'\overline Q-\overline Q'Q)\ .$

\medskip
\textbf{Step 3: Preservation by the Flow.} 
In order to prove the last part of Proposition \ref{prop:spectralNsoliton}, we introduce the infinitesimal generator of the adjoint Lax--Beurling semi--group, namely the operator $G$  such  that
$$S(\eta )^*={\rm e}^{-i\eta G}\ .$$
Notice that its operator domain is given by
$${\rm dom}(G)=\{ f\in L^2_+(\R) : \hat f _{\vert ]0,+\infty[}\in H^1(]0,+\infty [)\}$$
and that 
$$\widehat{(Gf)}(\xi )=i\frac{d\hat f}{d\xi } \quad \mbox{for}  \quad \xi >0.$$
In particular, the operator $G$  acts on $K_\theta $ for every finite Blaschke product $\theta $. We claim that properties (i) and (ii) are equivalent to the following statement:\\
(iii) The space $\mathcal{E}_{pp}(u)$ has dimension $N$, contains $u$, and moreover $u\in \mathrm{dom}(G) $ with $Gu\in \mathcal{E}_{pp}(u)$.

Indeed, as we just observed, if $u$ is a $N$-soliton, then it satisfies (iii). Conversely, assume that $u\in H^1_+$ satisfies (iii). We appeal to a corollary of Lemma \ref{lem:firstbracket}. 
 \begin{lem}\label{lem:secondbracket}
Let $f\in {\rm dom}(G)\cap H^1_+$ such that $L_uf\in {\rm dom}(G)$. Then $Gf\in H^1_+$ and
$$GL_uf-L_uGf=if-i\frac{\inner fu}{2\pi}u\ .$$
\end{lem}
\begin{proof}
For every $h\in {\rm dom}(G)$, we have
$$Gh=i\lim_{\eta \to 0^+}\frac{S(\eta)^*h-h}{\eta }\ .$$
Rewriting Lemma \ref{lem:firstbracket} as
$$\lim_{\eta \to 0}\left [\frac{S(\eta )^*-I}{\eta }, L_u \right ]f=f-\frac{1}{2\pi}\inner f u u. $$
Using that $L_u$ is a closed operator, the lemma follows.
\end{proof}
Applying Lemma \ref{lem:secondbracket} to $f=u$, we infer that $L_u(u)\in \mathrm{dom}(G)$ and that
$$GL_u(u)=L_u(Gu)+iu -i\frac{\| u\|^2_{L^2}}{2\pi}u\ .$$
In particular, $GL_u(u)\in \mathcal{E}_{pp}(u)$. Iterating this process, we conclude by an easy induction that, for every integer $k$,  $L_u^k u\in \mathrm{dom}(G)$  and that $GL_u^ku\in \mathcal{E}_{pp}(u)$. Now recall from Proposition \ref{prop:L_basic}  that $L_u$ has $N$ simple eigenvalues on $\mathcal E_{pp}(u)$, and that the component of $u$ on any eigenvector is different from $0$. Consequently, $u$ is a cyclic vector for $L_u$ in $\mathcal E_{pp}(u)$, namely the $N$ vectors $u,L_u,\cdots, L_u^{N-1}u $ form a basis of $\mathcal E_{pp}(u)$. From this we infer that 
$G$ acts on $\mathcal E_{pp}(u)$, and finally that $S(\eta )^* $ acts on $\mathcal E_{pp}(u)$, whence (i).\\
Let us prove that property (iii) is preserved by the flow of (CM-DNLS). Consider $u_0$ satisfying (iii), and denote by $u$ the solution of \eqref{eq:NLS} with $u(0)=u_0$, on its maximal time interval.  Notice that we know from (ii) that $u_0$ belongs to every $H^s$. Therefore, by the well-posedness result  Proposition \ref{prop:lwp_kato}, $u(t)$ belongs to every $H^s$, hence we do not  have to worry about its regularity. We are going to use the Lax equation provided by Lemma \ref{lem:Lax}. Denote by $U(t)$ the one-parameter family of unitary operators on $L^2_+(\R )$ defined as
$$\frac d{dt}U(t)= B_{u(t)}U(t)\ ,\ U(0)=I\ .$$
Then Lemma \ref{lem:Lax}  implies  
\begin{equation}\label{eq:conjugation}
L_{u(t)}=U(t)L_{u_0}U(t)^*\ .
\end{equation}
Consequently, $\mathcal {E}_{pp}(u(t))=U(t) [\mathcal{E}_{pp}(u_0)]$ has dimension $N$. Furthermore, in view of Lemma \ref{lem:hierarchy}, 
the spectral measure of $L_{u(t)}$ associated to the vector $u(t)$ is the same as the spectral measure of $L_{u_0}$ associated to the vector $u_0$. This implies that $\mathcal {E}_{pp}(u(t))$ contains $u(t)$. It remains to prove that $u(t)\in \mathrm{dom}(G)$ and that $Gu(t)\in \mathcal {E}_{pp}(u(t))$. 
Let us appeal to the reformulation of the dynamics as 
$$\partial_tu=\tilde B_uu\ ,$$
where $\tilde B_u=B_u-iL_u^2$ according to \eqref{def:B_tilde}. Given $\eta >0$, define
$$v(t,\eta ):=i\frac{S(\eta)^*u(t)-u(t)}{\eta }$$
and observe that 
$$\partial_tv(t,\eta )=\tilde B_uv(t,\eta )+g(t,\eta )\ ,\ g(t,\eta ):=i\left [\frac{S(\eta )^*}{\eta}, \tilde B_{u(t)} \right ]u(t)\ ,$$
which, in view of \eqref{eq:conjugation},  can be solved as 
$$U(t)^*v(t,\eta )={\rm e}^{-itL_{u_0}^2}v(0,\eta )+\int_0^t {\rm e}^{i(\tau -t)L_{u_0}^2}U(\tau)^*g(\tau,\eta )\, d\tau\ .$$
Using Lemma \ref{lem:firstbracketbis} and the expression of $B_u$, we  obtain
\begin{lem}\label{lem:thirdbracket}
If $u\in H^1_+(\R )$, we have, for every $f\in L^2_+(\R )$,
$$\lim_{\eta \to 0}i\left [\frac{S(\eta )^*}{\eta }, B_u \right ]f =\frac{1}{2\pi}(\inner f {L_uu}u +\inner fu L_uu)\ .$$
\end{lem}
Combining Lemma \ref{lem:thirdbracket} with Lemma \ref{lem:firstbracket}, we infer, locally uniformly in $t$, 
$$\lim_{\eta \to 0}g(t,\eta )=2 L_{u(t)}u(t)\ .$$
This shows that $v(t, \eta )$ has a limit $Gu(t)$ in $L^2_+$ as $\eta \to 0$, characterized by
\begin{eqnarray*}
U(t)^*Gu(t)&=&{\rm e}^{-itL_{u_0}^2}Gu_0+2\int_0^t {\rm e}^{i(\tau -t)L_{u_0}^2}L_{u_0} U(\tau)^*u(\tau )\, d\tau\\
&=&{\rm e}^{-itL_{u_0}^2}Gu_0+2t\, {\rm e}^{-itL_{u_0}^2}L_{u_0}u_0\ .
\end{eqnarray*}
Notice that, by (iii),  the right hand side of the above equation belongs to $\mathcal{E}_{pp}(u_0)$. Consequently, $Gu(t)\in U(t)[\mathcal{E}_{pp}(u_0)]=\mathcal{E}_{pp}(u(t))$. This completes the proof.
 \end{proof}

\begin{remark*}
In fact, one can  easily check that the operator $L_u$ restricted to the invariant subspace $\theta L^2_+=(K_\theta)^\perp$ has absolutely continuous simple spectrum with
$$ L_u(\theta h)=\theta Dh \quad \mbox{for all $h \in H^1_+(\R)$.} $$
\end{remark*}

For any $u \in H^1_+(\R)$, we notice that the operator $L_u$ has the essential spectrum $\sigma_{ess}(L_u) = [0, \infty)$. In the case of multi-solitons, we find that $0$ is always an embedded eigenvalue.

\begin{prop}\label{prop:kernel}
For any multi-soliton potential $u \in H^1_+(\R)$, we have that 
$$L_u(1-\theta )=0,$$
where $\theta(x) = \frac{\overline{Q}(x)}{Q(x)}$ with the notation from Proposition \ref{prop:spectralNsoliton} (ii) above.
\end{prop}
\begin{proof} We observe that
$$L_u(1-\theta )=-D\theta -u\Pi_+(\overline u(1-\theta))=-D\theta +u\Pi_+(\overline u\theta)\ .$$
Notice that the function
$$\overline u\theta =\frac{\overline P}{\overline Q}\frac{\overline Q}{Q}=\frac{\overline P}{Q}$$
belongs to $L^2_+(\R)$. Hence we can deduce
$$L_u(1-\theta )=-D\theta +u\overline u\theta =0,$$
because of the constraint in Proposition \ref{prop:spectralNsoliton} (ii).
\end{proof}

\section{Dynamics of Multi-Solitons}

This section is devoted to the study of multi-solitons, i.\,e., solutions $u(t,x)$ with initial datum giving by a multi-soliton potential $u_0 \in H^1_+(\R)$ (see Proposition \ref{prop:spectralNsoliton} above). By means of an inverse spectral formula, we will be able to prove global-in-time existence for all multi-solitons. This is a large data result which is beyond the scope of a-priori bounds. Second, we prove that all multi-solitons with $N \geq 2$ exhibit an energy cascade (growth of Sobolev norms) as $t \to \pm \infty$. 

\label{sec:dyn_sol}

\subsection{Preliminary Discussion} Let us first consider the following {\em pole ansatz} of the form
\be \label{eq:u_ansatz_sol}
u(t,x) = \sum_{j=1}^N \frac{a_j(t)}{x-z_j(t)} \in H^1_+(\R),
\ee   
where $a_1(t), \ldots, a_N(t) \in \C \setminus \{ 0 \}$ and pairwise distinct poles $z_1(t), \ldots, z_N(t)$ in the complex lower halfplane $\C_-$. If we plug this ansatz into (CM-DNLS), then a straightforward calculation shows that the self-consistency of \eqref{eq:u_ansatz_sol} leads to the set of nonlinear constraints given by
\be \label{eq:cm_constraints}
\sum_{j=1}^N \frac{a_j(t) \ov{a}_k(t)}{z_j(t) - \ov{z}_k(t)} = i \quad \mbox{for} \quad k=1, \ldots, N.
\ee
Note that these conditions have already appeared in the discussion of multi-soliton potentials (see Proposition \ref{prop:spectralNsoliton} above). Furthermore,  the equations of motions which govern the parameters $\{ a_j(t), z_j(t) \}_{j=1}^N$ are found to be
\be \label{eq:cm_one}
\dot{a}_k = 2 i \sum_{\ell \neq k}^N \frac{a_\ell - a_k}{(z_k-z_\ell)^2} \quad \mbox{and} \quad a_k \dot{z}_k = -2 i \sum_{\ell \neq k}^N \frac{a_\ell}{z_k - z_\ell}
\ee
with $k=1, \ldots, N$. A tedious calculation shows that the constraints \eqref{eq:cm_constraints} are indeed preserved by the time evolution determined by \eqref{eq:cm_one}. Finally, we remark that the first-order system \eqref{eq:cm_one} can be used to derive that
\be \label{eq:cm_two}
\ddot{z}_k = \sum_{\ell \neq k}^N \frac{8}{(z_k - z_\ell)^3} \quad \mbox{for} \quad k=1, \ldots, N,
\ee
which is again confirmed by a lengthy calculation that we omit here. We note that \eqref{eq:cm_two} can be regarded as a complexified version\footnote{That is, we formally generalize the positions $x_j \in \R$ to complex numbers $z_j \in \C_-$.} of the \textbf{Calogero--Moser (CM) system} on the real line, whose complete integrability was proven by J.~Moser in \cite{Mo-75} (see also \cite{OlPe-76, KaKoSt-78, OlPe-81}). Let us mention that the pole dynamics of rational solutions of completely integrable PDEs are governed by (complexified versions) of CM systems have also been found for the Benjamin-Ono, KdV and Half-Wave Maps equations in \cite{AiMcMo-77,Ca-79, BeKlLa-20}. 

However, we emphasize that working with the pole ansatz in \eqref{eq:u_ansatz_sol} leads to the following {\em caveats} that ultimately need to be addressed.

\begin{enumerate}
\item {\em Collision of poles:} It may happen that $z_j(t) \to z_k(t)$ for some pair $j \neq k$ as $t \to T$ with some finite time $T> 0$. From \eqref{eq:cm_two} we expect that the solution $z_k(t)$ blows up in $C^2$ as $t \to T$. But the solution $u(t,x)$ itself may stay smooth as $t \to T$, whereas the pole ansatz \eqref{eq:u_ansatz_sol} becomes invalid only. Explicit examples of pole collisions will be given in Subsection \ref{subsec:two_sol} below.

\item {\em Showing that $z_k(t) \in \C_-$:} The major step in showing global-in-existence for multi-solitons consists in proving that the poles $z_k(t)$ stay in the lower complex half-plane $\C_-$.   
\end{enumerate}

To systematically tackle the problems (1) and (2), we will derive an {\em inverse spectral formula} for multi-solitons, which entails the pole ansatz \eqref{eq:u_ansatz_sol} as a special case. Furthermore, the dynamical evolution of multi-solitons $u(t,x)$ will be encoded by the {\em linear flow} of a suitable matrix $\Ms(t) \in \C^{N \times N}$ such that
\be
\Ms(t) = 2 \Vs t + \Ws
\ee
with some constant matrices $\Vs$ and $\Ws$ in $\C^{N \times N}$; see Proposition \ref{prop:solution} and \eqref{eq:Ms}--\eqref{eq:VsWs} below. Moreover, we remark that solving classical CM systems on the real line by means of linear flows of $N \times N$-matrices was successfully used in \cite{OlPe-76}. Finally, notice that similar inverse formulae for multi-soliton solutions were derived in the case of the Benjamin--Ono equation in \cite{Su-21} and in the case of the cubic Szeg\H{o} equation in \cite{GePu-22}. However, in these two examples, the global well-posedness result was established directly by other methods, while it seems to be  the first time that such inverse formulae provide global existence.

\subsection{Inverse Spectral Formula and Time Evolution}

Let 
$$
u(x) = \frac{P(x)}{Q(x)} \in H^1_+(\R)
$$ 
be a multi-soliton potential. We use the notation introduced in the proof of Proposition \ref{prop:spectralNsoliton} with
$$
\theta(x) = \frac{\ov{Q}(x)}{Q(x)} \quad \mbox{and} \quad K_\theta = \frac{\C_{N-1}[x]}{Q(x)}.
$$
Recall that the Lax--Beurling semigroup $S(\eta )$ leaves the space $\theta L^2_+$ invariant, and hence its adjoint semigroup $S(\eta )^*$ leaves the subspace $K_\theta $ invariant.  This observation leads to the following formula, where $G$ is the operator introduced in subsection 5.2.
\begin{prop}\label{inverse}
For every $f\in K_\theta $, 
$$f(x)=\frac{1}{2\pi i}\inner {(G-xI)^{-1}f}{1-\theta} \quad \mbox{for} \quad {\rm Im}(x)>0\ .$$
\end{prop}
\begin{proof}
We start from the inversion Fourier formula for every element of $L^2_+(\R)$,
$$f(x)=\frac{1}{2\pi}\int_0^\infty \mathrm{e}^{ix\xi }\hat f(\xi )\, d\xi \ .$$
We claim that if $f\in K_\theta $ we have
\be\label{Fouriertheta}
\hat f(\xi )=\inner {S(\xi)^*f}{1-\theta}\ ,\ \xi >0\ .
\ee
Indeed, by the Plancherel theorem,
$$\hat f(\xi)=\lim_{\eps \to 0^+}\int_\R {\rm e}^{-ix\xi} \frac{f(x)}{1+i\eps x}\, dx=\lim_{\eps \to 0^+}\bbinner {S(\xi )^*f} {\frac{1}{1-i\eps x}}\ ,$$
since $(1-i\eps x)^{-1}\in L^2_+$. For the same reason $\theta (1-i\eps x)^{-1}\in \theta L^2_+$ which is orthogonal to $K_\theta $ and in particular to $S(\xi)^*f$. Therefore
$$ \hat f(\xi)=\lim_{\eps \to 0^+}\bbinner {S(\xi )^*f} {\frac{1-\theta}{1-i\eps x}}\ ,$$
which yields \eqref{Fouriertheta}. It remains to plug \eqref{Fouriertheta} into the inversion Fourier formula, reminding that $S(\xi )^*={\rm e}^{-i\xi G}$, and the proposition follows.
\end{proof}

Next, we are going to use Proposition \ref{inverse} in the particular case $f=u$. The interesting feature is that the inner product in this formula takes place in $K_\theta$, of which we can choose an orthonormal basis made of eigenvectors $(\psi_1, \ldots, \psi_N)$ of $L_u$ so that
$$L_u\psi_j=\lambda_j\psi_j \quad \mbox{for} \quad j=1,\dots ,N.$$
Here and throughout the following we label the eigenvalues such that 
$$\lambda_1=0, \lambda_2, \ldots, \lambda_N,$$ 
where we recall that $0$ is always an eigenvalue of $L_u$ by Proposition \ref{prop:kernel}. In view of Proposition \ref{prop:L_basic},  we can choose the normalisation
$$\inner u {\psi_j} =\sqrt{2\pi} \quad \mbox{for} \quad j=1,\dots, N.$$
Note that, because of Proposition \ref{prop:kernel}, it holds 
$$\inner {1-\theta} {\psi_j} =0 \quad \mbox{for} \quad  j=2,\dots ,N .$$

 Let us first discuss the case of $\inner {1-\theta} {\psi_1}$. Notice that every element $f$ of 
 $K_\theta $ is smooth on $]0,+\infty[$ with limits at $0^+$. In particular, we can pass to the limit in formula \eqref{Fouriertheta} as $\xi \to 0^+$ to obtain
 $$\hat f(0^+)=\inner f{1-\theta}\ .$$
 On the other hand, the identity $L_u \psi_1=0$ reads
 $$\xi \hat \psi_1(\xi )=\frac{1}{4\pi^2}\int_0^\xi \hat u(\xi-\zeta)\left (\int_0^\infty \overline{\hat u(\tau)}\hat \psi_1(\tau +\zeta)\, d\tau\right)\, d\zeta $$
 so that
 \be \label{hatpsi_00}
 \inner {\psi_1} {1-\theta}=\hat \psi_1(0^+)=\frac{\hat u(0^+)}{2\pi}\inner u{\psi_1}=\frac{\hat u(0^+)}{\sqrt{2\pi}}\ .
 \ee
 Coming back to the equation satisfied by $u$,
 \be \label{SzDNLS}
 i\partial_tu+\partial_x^2u+2D_+(|u|^2)u=0\ ,
 \ee
 and taking the Fourier transform, we observe that
 $$i\partial_t\hat u(t, \xi )-\xi^2\hat u(t,\xi )+2\int_0^{\infty}{\rm e}^{-ix\xi}D_+(|u|^2)(t,x)u(t,x)\, dx=0\ .$$
 Passing to the limit as $\xi \to 0^+$ and observing that $D_+(|u|^2)$ and $u$ both belong to $L^2_+$, we infer
 $$\partial_t\hat u(t,0^+)=0.$$
Hence $\hat u(0^+)$ is a conserved quantity and, consequently, the inner product product $\inner  {1-\theta} {\psi_1}$ is conserved as well. Let us come back to the inverse formula
 \be \label{inverseu}
 u(x)=\frac{1}{2i\pi}\inner {(G-xI)^{-1}u}{1-\theta}\ ,\ {\rm Im}(x)>0\ .
 \ee
 In the orthonormal basis $(\psi_1,\dots, \psi_{N})$ of $K_\theta $, we have just observed that the components of $u$ and of $1-\theta $ are conserved quantities. Next, let us discuss the matrix of $G$ in this basis, which we denote as
 $$\Ms:=(\inner {G\psi_k}{\psi_j})_{1 \leq j,k\leq N}\ .$$
 Of course, the matrix $\Ms=\Ms(t)$ will depend on $t$ as well through the evolution of the eigenfunctions $\psi_j$ as given by the Lax structure in order to keep $\inner u {\psi_j} $ constant,  i.\,e., we have $\dot{\psi}_j = \tilde B_u \psi_j$. Indeed, we will later use this to derive explicitly formulas for $\Ms$. For notational ease, we will sometimes omit the $t$-dependence of $\Ms(t)$.

 We first consider the case $j\ne k$. For this We recall that every element of $K_\theta $ belongs to ${\rm dom}(G)$. Then, for $j\ne k$, we observe that
$$(\lambda_k-\lambda_j)\inner {G\psi_k}{\psi_j}=\inner {(GL_u-L_uG)\psi_k}{\psi_j}=i\inner {\psi_k}{\psi_j}-i\frac{\inner {\psi_k}u \inner u {\psi_j}}{2\pi}=-i\ $$
so that
\be \label{offdiag}
\inner {G\psi_k}{\psi_j}=\frac{i}{\lambda_j-\lambda_k} \quad \mbox{if} \quad j \neq k.
\ee 
Finally, let us discuss the diagonal elements $\inner {G\psi_j}{\psi_j}$. Notice that their imaginary parts are easy to calculate, since
$${\rm Im} \, \inner {G\psi_j}{\psi_j}=\frac{1}{2\pi}{\rm Re} \, \bbinner {\frac{d\hat \psi_j}{d\xi}}{\hat \psi_j}=-\frac{|\hat\psi_j(0^+)|^2}{4\pi}=-\frac{|\inner {\psi_j}{1-\theta}|^2}{4\pi}$$
which is $0$ whenever $j \neq 1$. For $j=1$, we use  \eqref{hatpsi_00} to conclude 
$${\rm Im}\inner {G\psi_1}{\psi_1}=-\frac{|\hat u(0^+)|^2}{8\pi^2}\ .$$
As for the real part of the diagonal elements, we are going to compute their time derivatives if $u$ is a solution of \eqref{SzDNLS}. From the Lax pair formula, we may assume that 
$$\dot \psi_j=B_u\psi_j\ ,\ B_u=T_uT_{\partial_x\overline u}-T_{\partial_xu}T_{\overline u}+i(T_uT_{\overline u})^2\ .$$
Then
$$\frac{d}{dt}\inner {G\psi_j}{\psi_j}=\inner {[G,\tilde B_u]\psi_j}{\psi_j}=\inner {[G,B_u]\psi_j}{\psi_j}\ .$$
Passing to the limit in   Lemma \ref{lem:thirdbracket}, we have
$$[G,B_u]f=\frac{1}{2\pi}(\inner f{L_uu}u +\inner fu L_uu\ ).$$
Consequently, we get
\begin{eqnarray*}
\inner {[G,B_u]\psi_j}{\psi_j}
&=&\frac{1}{2\pi}\inner{L_u\psi_j}{u}\inner{u}{\psi_j}+\frac{1}{2\pi}\inner{\psi_j}{u}\inner{L_uu}{\psi_j}\\
&=&2\lambda_j\ .
\end{eqnarray*}
Summing up, we have proved that
\be \label{evoldiag}
\frac{d}{dt}\inner {G\psi_j}{\psi_j}=2\lambda_j\ .
\ee
The inverse spectral formula therefore reads, setting
$$\gamma_j:={\rm Re}(\Ms_{jj})\ ,\ \sqrt{2\rho}\, \mathrm{e}^{i\varphi }:=\frac{\hat u(0^+)}{2\pi i}\ ,\ \rho >0 .$$

The discussion above shows that the following results holds. 
\begin{prop}\label{prop:solution}
If $u(t) \in H^1_+(\R)$ is a multi-soliton potential such that $L_u$ has eigenvalues $\lambda_1=0,\lambda_2,\dots ,\lambda_{N}$, then $u$ can be recovered as
$$u(t,x)=\sqrt{2\rho}\, {\rm e}^{i\varphi }\inner{(\Ms(t)-xI)^{-1}X}{Y}_{\C^N} \quad \mbox{for} \quad {\rm Im} \, x \geq 0,$$
where
\begin{eqnarray*}
&&X:=  (1,\dots ,1)^T \ ,\ Y:= (1,0,\dots ,0)^T  \ ,\\
&& \Ms_{jk} =\frac{i}{\lambda_j-\lambda_k} \ \ (1\leq j\ne k\leq N), \quad \Ms_{jj}= \gamma_j-i\rho \delta_{j1} \ \ (j=1,\dots,N).
\end{eqnarray*}
Furthermore, the following evolution laws hold:
$$\frac{d}{dt}\varphi =0\ ,\frac{d}{dt}\rho =0\ ,\ \frac{d}{dt}\gamma_j=2\lambda_j, \quad (j=1,\dots ,N).$$
\end{prop}

\subsection{Global-in-Time Existence}
In order to prove that $N$-soliton solutions $u(t,x)$ can be extended to all times $t \in \R$, we are going to study the eigenvalues of the matrix $\Ms(t) \in \C^{N \times N}$ introduced above. To this end, we note that the time evolution of $\Ms(t)$ stated in Proposition \ref{prop:solution} can be written as
\be \label{eq:Ms}
\Ms(t) = 2 \Vs t + \Ws
\ee
with the constant complex $N \times N$-matrices $\Vs=(\Vs_{jk})_{1 \leq j,k \leq N}$ and $\Ws=(\Ws_{jk})_{1 \leq j,k \leq N}$ having the entries:
\be \label{eq:VsWs}
\Vs_{jk} = \lambda_j \delta_{jk}, \quad \Ws_{jk} = \begin{dcases*}  \gamma_j - i \rho \delta_{j1} & if $j=k$, \\ \frac{i}{\lambda_j-\lambda_k} & if $j \neq k$. \end{dcases*} 
\ee
Recall  that $\lambda_1=0, \lambda_2, \ldots, \lambda_N \in \R$ are real and pairwise distinct and $\rho > 0$ is a strictly positive real number, whereas $\gamma_1, \ldots, \gamma_N \in \R$ are real numbers (which are not necessarily pairwise distinct).

To prove that multi-soliton solutions extend to all times $t \in \R$, we show that all the eigenvalues of $\Ms(t)$ are always in the lower complex plane $\C_- = \{ z \in \C : \mathrm{Im} \, z < 0 \}$.

\begin{lem} \label{lem:M_lower}
For any $t \in \R$, all eigenvalues of the matrix $\Ms(t)$ have strictly negative imaginary parts, i.\,e., it holds
$$
\sigma(\Ms(t)) \subset \C_{-} \quad \mbox{for $t \in \R$}.
$$
\end{lem}

\begin{remark*}
Below we will prove the remarkable fact that all the eigenvalues of $\Ms(t)$ except for one will asymptotically converge to the real axis as $t \to \pm \infty$. As a consequence, this implies that all $N$-solitons with $N \geq 2$ will have an algebraic growth of their Sobolev norms according to $\| u(t) \|_{H^s} \sim |t|^{2s}$ for any $s>0$.
\end{remark*}

\begin{proof}
Let $t \in \R$ be given. For notational convenience we write $\Ms=\Ms(t)$ in what follows. We readily verify the identities
\be \label{eq:M_master}
\frac{(\Ms-\Ms^*)_{jk}}{2 i} = -2 \rho \delta_{1j} \delta_{1k}, \quad [\Ms, \Vs] = [\Ws, \Vs] =  i I - i \langle \cdot, X \rangle X,
\ee
where $\langle \cdot, \cdot \rangle_{\C^N}$ denotes the standard inner product on $\C^N$ and $X = (1, \ldots, 1)^T \in \C^N$. Because of the first identity above, any eigenvalue $z \in \sigma(\Ms)$ must satisfy 
$$
\mathrm{Im} \, z \leq 0.
$$ 
Let us prove that $z \not \in \R$ holds. We argue by contradiction as follows. Suppose that $z \in \R$ and let $v = (v_1, \ldots, v_{N})^T \in \C^N \setminus \{0 \}$ be a corresponding eigenvector, i.\,e.,
$$
\Ms v = z v.
$$
Since we assume that $z$ is real, we can use the first identity in \eqref{eq:M_master} to deduce that
$$
0 = \frac{1}{2i} \langle (\Ms-\Ms^*) v,v \rangle _{\C^N}= -2 \rho |v_1|^2
$$
and therefore $v_1=0$. Projecting the equation $\Ms v = z v$ onto the 1st mode and recalling that $\lambda_1=0$ by assumption, we infer
$$
\sum_{j=2}^{N} \frac{v_j}{\lambda_j} = 0.
$$
This can be written as
$$
\langle w, X \rangle = 0 \quad \mbox{with} \quad w := \left ( 0, \frac{v_2}{\lambda_2}, \ldots, \frac{v_{N}}{\lambda_{N}} \right )^T.
$$
Noticing that $\Vs w = v$ and applying the second identity in \eqref{eq:M_master}, we find
\be \label{eq:M_master2}
\Vs(\Ms-z I) w = i ( w - \langle w, X \rangle _{\C^N}X) = i w.
\ee
If we take inner product of the left side with $w$ and using that $\Vs=\Vs^*$, we conclude
\begin{align*}
\langle w, \Vs (\Ms- z I) w \rangle _{\C^N}& = \langle \Vs w, \Ms w \rangle _{\C^N}-z \langle w, \Vs w \rangle _{\C^N}= \langle \Ms^* v, w \rangle _{\C^N}- z \langle w, v \rangle _{\C^N}\\
& = z  \langle v, w \rangle_{\C^N} - z \langle w, v \rangle_{\C^N} = z \sum_{j=2}^{N} \frac{|v_j|^2}{\lambda_j} - z \sum_{j=2}^{N} \frac{|v_j|^2}{\lambda_j} = 0,
\end{align*}
where we used that $z \in \R$ and that $\Ms^* v = \Ms v = z v$ holds thanks to the form of $\Ms$ and $v_1=0$. Thus from \eqref{eq:M_master2} we deduce that $0 = i |w|^2$. This shows that $w=0$ and consequently we find $v=0$, which is a contradiction. Therefore $z \in \R$ cannot be an eigenvalue of $\Ms$.
\end{proof}

As a direct consequence of Proposition \ref{prop:solution} and Lemma \ref{lem:M_lower}, we deduce the following result.

\begin{thm}[Global-in-Time Existence] \label{thm:gwp_N_soliton}
Suppose $u_0 \in H^1_+(\R)$ is an $N$-soliton potential with some $N \geq 1$. Then the corresponding solution $u(t)$ of \eqref{eq:NLS} with $u(0)=u_0$ extends to all times $t \in \R$.
\end{thm}

\begin{remark} The previous study shows that, to every $N$--soliton $u$, one can associate $$\Lambda (u):=(\varphi , \rho, \lambda_2,\dots ,\lambda_N,\gamma_1,\dots ,\gamma_N)\in \T \times (0,+\infty )\times \R^{2N-1}$$
with the condition that $\lambda_1=0,\lambda_2,\dots, \lambda_N$ are pairwise distinct, so that 
$$\sqrt{2\rho}\mathrm{e}^{i\varphi} =\frac{\hat u(0^+)}{2i\pi}$$ and $\lambda_1=0,\lambda_2, \dots,\lambda_N$ are the eigenvalues of $L_u$, while 
$\gamma_j=\inner {G\psi_j}{\psi_j}-i\rho \delta_{j1}$ if $L_u\psi_j=\lambda_j\psi_j, \Vert \psi_j\Vert_{L^2}=1$. \\
We claim that the spectral mapping $u\mapsto \Lambda (u)$ is surjective. Indeed, assume we are given $(\varphi , \rho, \lambda_2,\dots ,\lambda_N,\gamma_1,\dots ,\gamma_N)\in \T \times (0,+\infty )\times \R^{2N-1}$
with the condition that $\lambda_1=0,\lambda_2,\dots, \lambda_N$ are pairwise distinct, and consider the associated matrix $\Ms$. According to Lemma \ref{lem:M_lower}, we already 
know that  the eigenvalues of $\Ms$ belong to the lower half plane, so that we may define $u$ as
$$u(x)=\sqrt{2\rho}\, {\rm e}^{i\varphi }\inner{(\Ms-xI)^{-1}X}{Y}_{\C^N} \quad \mbox{for} \quad {\rm Im} \, x \geq 0.$$
We claim that $u$ is a $N$--soliton with $Q(x)=\det(xI-\Ms)$. Indeed, this is equivalent to the identity
\be\label{eq:Nsol}
|u(x)|^2=\frac{1}{i}\left (\frac{\ov Q'(x)}{\ov Q(x)}-\frac{Q'(x)}{Q(x)}\right )\ ,\ x\in \R\ .
\ee
Using the expressions of $u(x)$ and $Q(x)$, \eqref{eq:Nsol} is equivalent to
$$\vert \inner {Z(x)}X _{\C^N}\vert ^2=\Vert Z(x)\Vert ^2$$
where $Z(x):=(xI-\Ms ^*)^{-1}Y\ .$ The latter identity can be proved as follows. From $(xI-\Ms^*)Z(x)=Y$, we infer
$\Vs (xI-\Ms^*)Z(x)=0\ .$
Taking the imaginary part of the inner product of both sides with $Z(x)$, we obtain
\begin{eqnarray*}
0&=&\inner {(\Vs (xI-\Ms^*)-(xI-\Ms)\Vs) Z(x)}{Z(x)}_{\C^N}\\
&=&\inner {[\Vs ,(xI-\Ms^*)]Z(x)}{Z(x)}_{\C^N}=i(\vert \inner {Z(x)}{X}_{\C^N}\vert ^2-\Vert Z(x)\Vert^2)\ .
\end{eqnarray*}
Furthermore, from Proposition \ref{prop:solution}, the value of $\hat u(0^+)$ can be obtained by identifying the  coefficient of $1/x$ in the expansion of $u(x)$ as $x\to \infty $. In order to complete the proof of the surjectivity, we  just have to check the following identities, 
$$L_u\psi_j=\lambda_j\psi_j\ ,\ \inner {G\psi_j}{\psi_j}=(\gamma_j -i\rho \delta-{j1})\Vert \psi_j\Vert_{L^2}^2\ ,$$ 
if we define $\psi_j$ as $\psi_j(x)=\inner{(\Ms -xI)^{-1}Y_j}Y _{\C^N}\ ,$
where $Y_j$ denotes the column with $1$ on the line $j$ and $0$ on the other lines. This can be done by direct calculations. For instance, since $\Vs Y_j=\lambda_jY_j$ and $\Vs Y=0$, 
\begin{eqnarray*}
\lambda_j\psi_j(x)&=&\inner {(\Ms -xI)^{-1}\Vs Y_j}Y_{\C^N}=\inner {(\Ms -xI)^{-1}[\Vs ,\Ms](\Ms -xI)^{-1}Y_j}Y_{\C^N}\\
&=&D\psi_j(x)+i\inner{(\Ms-xI)^{-1}Y_j}{X}_{\C^N}\inner{(\Ms-xI)^{-1}X}{Y}_{\C^N}
\end{eqnarray*}
while 
\begin{eqnarray*}
&&\inner{ (\ov \Ms -xI)^{-1}X }{Y}_{\C^N}\inner{ (\Ms -xI)^{-1}Y_j }{Y}_{\C^N}=\inner{(\Ms ^*-xI)^{-1}YY^T(\Ms -xI)^{-1}Y_j}{X}_{\C^N}\\
&&=(2i\rho)^{-1}[\inner{(\Ms-xI)^{-1}X}{Y}_{\C^N}-\inner{(\Ms^*-xI)^{-1}Y_j}{X}_{\C^N}], 
\end{eqnarray*}
so that $i\inner{(\Ms-xI)^{-1}Y_j}{X}_{\C^N}\inner{(\Ms-xI)^{-1}X}{Y}_{\C^N}=-u(x)\Pi_+(\ov u\psi_j)(x)$. 
\end{remark}

Below  we will show that all multi-soliton solutions with $N \geq 2$ exhibit growth of Sobolev norms such that $(\phi , \rho, \lambda_2,\dots ,\lambda_N,\gamma_1,\dots ,\gamma_N)\in \T \times (0,+\infty )\times \R^{2N-1}$ is given so that 
$$
\|u(t) \|_{H^s} \sim |t|^{2s} \quad \mbox{as $t \to \pm \infty$ for any $s > 0$}.
$$
This demonstrates that global-in-time existence for such $u(t,x)$ cannot be inferred from a-priori bounds in $H^1$. Furthermore, it shows that the infinite hierarchy of conversation laws $\{ I_k(u) \}_{k=0}^\infty$ generally fail  to produce a-priori bounds on solutions of (CM-DNLS) in the large data regime with $\|u(0) \|_{L^2}^2 > 2 \pi$. 

\subsection{Explicit Example: Two-Soliton Solutions}

\label{subsec:two_sol}

Before we study the case of $N$-solitons with arbitrary $N \geq 1$, it is instructive to first consider the case $N=2$ in detail. In this case, the constraints in Proposition \ref{prop:spectralNsoliton} 
can be solved explicitly. Suppose that
\be
u(x) = \frac{P(x)}{Q(x)} \in H^1_+(\R)
\ee
is a two-soliton potential, i.\,e.,~$Q \in \C[x]$ is a polynomial of degree 2 with zeros in $\C_-$ and $P\in \C[x]$ is a polynomial of degree at most 1 satisfying the condition stated in Proposition \ref{prop:spectralNsoliton} above.

We start with generic case by assuming that $Q(x)$ has two different zeros $z_1 \neq z_2$ in $\C_-$. Then condition (ii) in Proposition \ref{prop:spectralNsoliton} implies that the polynomial $P(x)=\mbox{const}$ is constant and we find that
$$
u(x) = \frac{a_1}{x-x_1} + \frac{a_2}{x-z_2}, \quad \sum_{j=1}^2 \frac{a_j\overline a_k}{z_j-\overline z_k}=\ii \quad \mbox{with $k=1,2$} .
$$
The constraints on $a_1, a_2 \in \C$ with given $z_1, z_2 \in \C_-$ can be solved explicitly as follows.  Writing for convenience $y_j=-{\rm Im}z_j>0$, we find 
$$\frac{|a_1|^2}{2y_1}-i\frac{a_2\overline a_1}{z_2-\overline z_1}=1\ ,\ -i\frac{a_1\overline a_2}{z_1-\overline z_2}+\frac{|a_2|^2}{2y_2}=1\ .$$
Let us set 
$$\xi :=\frac{|a_1|^2}{2y_1}=\frac{|a_2|^2}{2y_2}\ ,\ \eta :=-i\frac{a_2\overline a_1}{z_2-\overline z_1}=-i\frac{a_1\overline a_2}{z_1-\overline z_2}=1-\xi \ .$$
Then 
$$a_j=\sqrt{2\xi y_j}\mathrm{e}^{i\theta_j}$$
and 
$$1-\xi =\eta =-2i\sqrt{y_1y_2}\, \xi \, \frac{\mathrm{e}^{i(\theta_1-\theta_2)}}{z_1-\overline z_2}\ .$$
Now we discuss the two solutions, according to $\xi <1$ or $\xi >1$. If $\xi <1$,
$$1-\xi =\xi \frac{2\sqrt{y_1y_2} }{|z_1-\overline z_2|}\ ,$$
hence
$$\xi =\frac{1}{1+\frac{2\sqrt{y_1y_2} }{|z_1-\overline z_2|}}.$$
Thus we get
$$
a_1 = i\left ( \frac{2y_1}{1+2\frac{\sqrt{y_1y_2}}{|z_1-\overline z_2|} }\right )^{\frac 12}\frac{z_1-\overline z_2}{|z_1-\overline z_2|} \mathrm{e}^{i\theta }\ , \quad a_2 = \left ( \frac{2y_2}{1+2\frac{\sqrt{y_1y_2}}{|z_1-\overline z_2|} }\right )^{\frac 12}{\rm e}^{i\theta }\ .
$$
 If $\xi >1$, 
 $$\xi -1=\xi \frac{2\sqrt{y_1y_2} }{|z_1-\overline z_2|}\ ,$$
 hence
 $$\xi =\frac{1}{1-\frac{2\sqrt{y_1y_2} }{|z_1-\overline z_2|}}.$$
Therefore,
 $$
a_1 = -i\left ( \frac{2y_1}{1-2\frac{\sqrt{y_1y_2}}{|z_1-\overline z_2|} }\right )^{\frac 12}\frac{z_1-\overline z_2}{|z_1-\overline z_2|}\mathrm{e}^{i\theta }\ , \quad a_2 = \left ( \frac{2y_2}{1-2\frac{\sqrt{y_1y_2}}{|z_1-\overline z_2|} }\right )^{\frac 12}\mathrm{e}^{i\theta }\ .
$$
 Let us compute the remaining eigenvalue $\lambda$ of $L_u$ in both cases. It is enough which is the trace of $L_u$ on $\mathcal{E}_{\rm pp}=K_\theta$ in view of Proposition \ref{prop:kernel}. If $\{j,k\}=\{ 1,2\}$, one checks that
 $$L_u\left ( \frac{1}{x-z_k}  \right )=-i\frac{a_j}{a_k(z_k-z_j)}\frac{1}{x-z_k}+i\frac{a_j}{a_k(z_k-z_j)}\frac{1}{x-z_j}\ .$$
 This implies
 $$\lambda =-i\frac{a_2}{a_1(z_1-z_2)}-i\frac{a_1}{a_2(z_2-z_1)}.$$
After some calculations we obtain
 $$\lambda =\frac{-(y_1+y_2)}{\sqrt{y_1y_2}|z_1-\overline z_2|}<0 \quad \mbox{and} \quad \lambda =\frac{(y_1+y_2)}{\sqrt{y_1y_2}|z_1-\overline z_2|}>0$$
 in the first and in the second case, respectively.

Finally, we consider the non-generic case when $z_1=z_2=z$ and $Q(x)=(x-z)^2$. 
 Then
 $$F(x):=i(Q'(x)\overline Q(x)-Q(x)\overline Q'(x))=4y(x- z)(x-\overline z).$$
 Writing $y=-\mathrm{Im} \, z>0$ again, we have two possibilities:
 $$P(x)=\sqrt{4y}(x-z)\ ,\ u(x)=\frac{\sqrt{4y}}{x-z}$$
 and
 $$P(x)=\sqrt{4y}(x-\overline z)\ ,\ u(x)=\frac{\sqrt{4y}(x-\overline z)}{(x-z)^2}\ .$$
 Notice that the first case should not be confused with the case $N=1$, where
 $$Q(x)=x-z\ ,\ u(x)=\frac{\sqrt{2y}}{x-z}\ .$$
 In the first case, we have 
 $$L_u\left ( \frac{1}{(x-z)^2}  \right )=-\frac{1}{y}\frac{1}{(x-z)^2},$$
 so that $\lambda =-y^{-1}$, which can obtained from the first formula by letting $z_1\to z$ and $z_2\to z$. Similarly, in the second case, one obtains that $\lambda =y^{-1}$.
 
 As a next step, we study the time evolution of two-solitons. For $N=2$, the matrix $\Ms(t)$ is given by
\be 
\Ms(t)=\left (\begin{array}{cc} \gamma_1-i\rho& -i\lambda^{-1}  \\ i\lambda^{-1} &\gamma_2 + 2 \lambda t \end{array}   \right )
\ee
with some positive number $\rho > 0$ and real numbers $\gamma_1, \gamma_2 \in \R$ and $\lambda \neq 0$ denotes the non-zero eigenvalue of $L_u$. From Proposition \ref{prop:solution} we deduce 
\be \label{2soliton}
u(t,x)=\frac{\mathrm{e}^{i\varphi}\sqrt{2\rho}( \gamma_2 + 2\lambda t +i\lambda^{-1}-x )     }{ x^2-(\gamma_1 -i\rho +\gamma_2 + 2 \lambda t)x+(\gamma_1-i\rho)(\gamma_2 +2 \lambda t) -\lambda^{-2} }
\ee
with some constant $\varphi \in [0,2 \pi)$. The discriminant of the denominator of $u(t,x)$ is
\be 
\Delta(t) =(\gamma_1-i\rho -\gamma_2 - 2 \lambda t)^2+4\lambda^{-2}. 
\ee
Note that $\Delta(t)=0$ if and only if $\gamma_1=\gamma_2 + 2 \lambda t$ and $\rho |\lambda|= 2$. This corresponds to the degenerate two-solitons, the cases 1 and 2 occurring according to the sign of $\lambda$. Notice that, if $\rho |\lambda |=2$, the two-soliton solution of \eqref{SzDNLS}  will be degenerate at exactly one time $t\in \R$ characterized by  $\gamma_1=\gamma_2+2\lambda t$.

Finally, let us study the large-time behaviour of a two-soliton. Let $z_+(t)$ and $z_-(t)$ denote the poles of $u(t,x)$ at time $t$. We see that
\be 
z_\pm (t)=\frac 12\left ( \gamma_0-i\rho +\gamma_1 +2\lambda t\pm \sqrt{(\gamma_0-i\rho -\gamma_1-2\lambda t)^2+4\lambda^{-2}}\right ).
\ee
As $t \to \pm \infty $, we obtain
\be \label{asymptoticpoles}
z_+(t)\to \gamma_0-i\rho,\quad {\rm Re} \, z_-(t)=2\lambda t+O(1), \quad {\rm Im} \, z_-(t) = \frac{-\rho}{4\lambda^4t^2} + O(t^{-3}) .
\ee
The vanishing of ${\rm Im} \, z_-(t)$ as $t \to \pm \infty$ implies growth of the Sobolev norms for the two-soliton solution. More precisely, we claim that
\be \label{eq:sob_grow_2}
\| u(t)\|_{H^s}\sim c_s|t|^{2s} \quad \mbox{as $t\to \pm \infty$ for any $s > 0$}.
\ee
To prove this, we note that $z_+(t) \neq z_-(t)$ for all $t \in \R$ (except for one time $t$ at most) and we have that 
$$
u(t,x) = \frac{a_+(t)}{x-z_+(t)}+\frac{a_-(t)}{x-z_-(t)}, \quad  a_\pm(t)= \sqrt{2\rho}\mathrm{e}^{i\varphi}\frac{\gamma_1+2\lambda t+i\lambda^{-1}-z_\pm(t)}{\sqrt{\Delta (t)}}.
$$ 
From \eqref{asymptoticpoles}, we infer  as $|t| \to \infty$ that
$$a_+(t)=O(1) \quad \mbox{and} \quad a_-(t) \sim \frac{1}{|t|}.$$
Since
$$\hat u(t,\xi )=-2\pi i \left (a_+(t)\mathrm{e}^{-iz_+(t)\xi}+a_-(t)\mathrm{e}^{-iz_-(t)\xi}\right ) \quad \mbox{for} \quad \xi >0,$$
we deduce the bound \eqref{eq:sob_grow_2} from \eqref{asymptoticpoles} and by direct calculation.

\subsection{Long-Time Asymptotics}

We now study the long-time behavior for $N$-solitons with general $N \geq 2$. The key ingredient for the general understanding is the following result about the long-time asymptotics for the eigenvalues of the matrix $\Ms(t)$ in \eqref{eq:Ms}.

\begin{lem} \label{lem:z_largetime}
There exists $T_0 > 0$ sufficiently large such that all eigenvalues 
$$
\{ z_1(t), \ldots, z_{N}(t) \} \subset \C_-
$$
of $\Ms(t)$ are simple for $|t| \geq T_0$. As $|t| \to +\infty$, we have the asymptotic expansions:
$$
\mathrm{Re} \, z_k(t) = 2 \lambda_k t + \gamma_k + O(t^{-1}) \quad \mbox{for $k=1, \ldots, N$}
$$
$$
\mathrm{Im} \, z_1(t) = -\rho + O(t^{-1}), \quad \mathrm{Im} \, z_k(t) = -\frac{ \rho}{4 \lambda_k^4 t^2}  + O(t^{-3}) \quad \mbox{for $k=2, \ldots, N$}.
$$
\end{lem}

\begin{proof}
We use standard eigenvalue perturbation theory for 
$$
\As(\eps) = \As + \eps \Bs
$$ 
with a small parameter $|\eps| \ll 1$. Here $\As=\As^* \in \C^{N\times N}$ is a Hermitian matrix and $\Bs$ denotes an arbitrary matrix in $\C^{N \times N}$. 

Since $\Ms(t) = t ( 2 \Vs + t^{-1} \Ws)$ and by taking $\eps = t^{-1}$, it suffices to study the eigenvalues of
$$
\As(\eps) = \As + \eps \Bs \quad \mbox{with} \quad \mbox{$\As = 2 \Vs$ and $\Bs = \Ws$},
$$
where the constant matrices $\Vs$ and $\Ws$ are displayed in \eqref{eq:VsWs}. Because $\As=2 \mathrm{diag}(\lambda_1, \ldots, \lambda_{N})$ has $N$ simple eigenvalues, standard perturbation theory yields that $\As(\eps)$ has $N$ simple eigenvalue provided that $|\eps| \ll 1$ is sufficiently small. Thus for $|t|\geq T_0$ with $T_0 >0$ sufficiently large, we see that $\Ms(t)=t \As(t^{-1})$ has $N$ simple eigenvalues $\{z_1(t), \ldots, z_N(t) \}$ which all belong to $\C_-$ by Lemma \ref{lem:M_lower}.

The derivation of the claimed asymptotics for the eigenvalues requires an expansion up to order $\eps^3 = t^{-3}$. Let $\{ \mu_1(\eps), \ldots, \mu_{N}(\eps) \}$ denote the eigenvalues of $\As+\eps \Bs$, which are simple for $|\eps| \ll 1$ sufficiently small. From \cite{ReSi-78}[Section XII.1] we recall that
$$
\mu_k(\eps) = \mu_k + \mu_k^{(1)} \eps + \mu_k^{(2)} \eps^2 + \mu_k^{(3)} \eps^3 + O(\eps^4), \quad \mbox{with $\mu_k = 2 \lambda_k$ }.
$$
The first-order coefficient is the well-known expression 
$$
\mu_k^{(1)} = \Bs_{kk} = \gamma_k -i \rho \delta_{k1} .
$$
Next, the second-order contribution is purely real and given by
$$
\mu_k^{(2)} = -\sum_{j =1, j \neq k}^{N} \frac{1}{\mu_j - \mu_k} \Bs_{kj} \Bs_{jk} = -\frac{1}{2} \sum_{j=1, j \neq k}^{N} \frac{1}{(\mu_j - \mu_k)^3} \in \R.
$$
Finally, the third-order term reads
$$
\mu_k^{(3)} = \sum_{j \neq k, \ell \neq k}^{N} \frac{1}{(\mu_j-\mu_k)(\mu_\ell-\mu_k)} \Bs_{kj} \Bs_{j \ell} \Bs_{\ell k} - \sum_{j \neq k}^{N} \frac{1}{(\mu_j-\mu_k)^2} \Bs_{kj} \Bs_{jk} \Bs_{kk} .
$$
For the proof of the lemma, it suffices to determine $\mathrm{Im} \, \mu_k^{(3)}$ for $k=2, \ldots, N$. Since $\Bs_{kk} \in \R$ for $k \geq 2$ and $\Bs_{kj} \Bs_{jk} \in \R$ if $j \neq k$, we see
$$
\mathrm{Im} \sum_{j \neq k}^{N} \frac{1}{(\mu_j-\mu_k)^2} \Bs_{kj} \Bs_{jk} \Bs_{kk} = 0 \quad \mbox{for $k=2, \ldots, N$}.
$$
As for the first sum in the expression for $\mu_k^{(3)}$, we notice the symmetry property
$$
\Bs_{kj} \Bs_{j \ell} \Bs_{\ell k} = -\Bs_{\ell j} \Bs_{\ell j} \Bs_{j k} \quad \mbox{for $j \neq \ell$}.
$$
Thus we only need to consider the diagonal case when $j = \ell$. This leads to
$$
\mathrm{Im} \, \mu_k^{(3)} = \mathrm{Im} \sum_{j \neq k}^{N} \frac{1}{(\mu_j-\mu_k)^2} \Bs_{kj} \Bs_{jj} \Bs_{j k} .
$$ 
Recalling that $\Bs_{kj} \Bs_{jk} \in \R$ for $j \neq k$ and $\Bs_{jj} \in \R$ if $j \geq 2$, we deduce 
$$
\mathrm{Im} \, \mu_k^{(3)} = \mathrm{Im} \frac{1}{(\mu_1-\mu_k)^2} \Bs_{1k} \Bs_{11} \Bs_{k1}  = -\frac{\rho}{4 \lambda_k^4} \quad \mbox{for $k=2, \ldots, N$},
$$
using that $\mu_k= 2 \lambda_k$ and $\mu_1=2 \lambda_1=0$. Since $\Ms(t) = t^{-1} \As(t^{-1})$, we obtain the claimed asymptotic formulae.
\end{proof}

As a consequence of the preceding lemma, we obtain the following result.

\begin{prop} \label{prop:sol_large_time}
Let $u(t,x)$ be an $N$-soliton solution. Then there exists $T_0 > 0$ sufficiently large such that
$$
u(t,x) = \sum_{j=1}^N \frac{a_j(t)}{x-z_j(t)} \quad \mbox{for $|t| \geq T_0$},
$$
where $\{ z_1(t), \ldots, z_N(t) \} \subset \C_-$ denote the simple eigenvalues of $\Ms(t)$ and  coefficients $a_1(t), \ldots, a_N(t) \in \C \setminus \{ 0 \}$. 
\end{prop}

\begin{proof}
By Lemma \ref{lem:z_largetime}, the eigenvalues $\{ z_1(t), \ldots, z_N(t) \} \subset \C_-$ of $\Ms(t)$ are simple whenever $|t| \geq T_0$, where $T_0 > 0$ is some sufficiently large constant. 

Fix some  time $t \in \R$ with $|t| \geq T_0$. For notational convenience, we will omit the dependence of $u(t,x)$ and $\Ms(t)$ for the rest of the proof. Since $u \in H^1_+(\R)$ is a multi-soliton potential, we have
$$
u(x) = \frac{P(x)}{Q(x)}
$$
with some polynomial $Q \in \C[x]$ of degree $N$ having all its zeros in $\C_-$ and some polynomial $P \in \C_{N-1}[x]$ satisfying the condition in Proposition \ref{prop:spectralNsoliton}. Recall that $\Ms$ denotes the matrix (with respect to a suitable orthonormalbasis) of the operator $G$ acting on the invariant space $K_\theta= \frac{\C_{N-1}[x]}{Q(x)}$. From \cite{Su-21}[Lemma 3.3] we observe that $Q(x)$ is the characteristic polynomial of $G$ acting on $K_\theta$. Hence we conclude
$$
Q(x) = \det (xI  - \Ms).
$$
Since $\Ms$ has only simple eigenvalues, we find $Q(x) = \prod_{j=1}^N (x-z_j)$ with pairwise distinct zeros $z_j \in \C_-$. Since $P \ov{P} = i (Q' \ov{Q}-\ov{Q}' Q)$, we see that $P \in \C_{N-1}[x]$ has no common zeros with $Q(x)$. Thus, by partial fraction expansion, we conclude that
$$
u(x) = \frac{P(x)}{Q(x)} = \sum_{j=1}^N \frac{a_j}{x-z_j}
$$
with some  $a_1, \ldots, a_N \in \C \setminus \{0 \}$.
\end{proof}

\subsection{Growth of Sobolev Norms: Proof of Theorem \ref{thm:growth_intro}}
Let $N \geq 2$ and suppose $u(t,x)$ is an $N$-soliton, which by Theorem \ref{thm:gwp_N_soliton} exists for all times $t \in \R$. We claim that
\be  \label{eq:u_sob}
\|u(t) \|_{H^s} \sim_s t^{2s} \quad \mbox{as} |t| \to +\infty
\ee
for any $s >0$. 

For convenience, we discuss the limit $t \to +\infty$. (The case $t \to -\infty$ follows by the same reasoning.) We divide the proof into following steps.

\medskip
\textbf{Step 1.} We first derive the asymptotics for $a_1(t), \ldots, a_N(t)$. From Proposition \ref{prop:sol_large_time} we deduce that
\be \label{eq:u_sol_asy}
u(t,x) = \sum_{j=1}^N \frac{a_j(t)}{x-z_j(t)} \quad \mbox{for} \quad t \geq T_0
\ee
with some sufficiently large constant $T_0> 0$. Now we are in the position to use the system of differential equations in \eqref{eq:cm_one} for the time evolution of $a_1(t), \ldots, a_N(t)$. We claim that
\be \label{eq:a_asymp}
a_k(t) = \frac{\alpha}{\lambda_k^2 t} + O(t^{-2}) \quad \mbox{for $k=2, \ldots, N$}
\ee
with some non-zero constant $\alpha \in \C \setminus \{ 0 \}$. Indeed, using that $\dot{z}_k = 2 \lambda_k + o(1)$ as $t \to +\infty$ by Lemma \ref{lem:z_largetime}, we deduce from \eqref{eq:cm_one} that
$$
\vec{a}(t) = A(t) \vec{a}(t) + \vec{b}(t)
$$
where $\vec{a}(t)=(a_2(t), \ldots, a_N(t)) \in \C^{N-1}$ and $A(t) = (A(t))_{2 \leq k, \ell \leq N} \in \C^{N-1} \times \C^{N-1}$, $\vec{b}(t)=(b_2(t), \ldots, b_N(t)) \in \C^{N-1}$ are given by
$$
(A(t))_{k \ell} = \frac{-i}{(\lambda_k + o(1)) (z_k - z_\ell)} = O(t^{-1}) \quad \mbox{for $k \neq \ell$}, \quad (A(t))_{kk} = 0,
$$
$$
b_k(t) = \frac{-i a_0(t)}{(\lambda_k + o(1)) ( z_k - z_1)} = O(t^{-1}) a_1(t) .
$$
Here we also used that $|z_k(t)-z_\ell(t)| \sim t$ as $|t| \to \infty$ which follows from Lemma \ref{lem:z_largetime} together with the fact that all $\lambda_k \neq \lambda_\ell$ for $k \neq \ell$. Next, as a direct consequence of \eqref{eq:cm_one}, we infer the conservation law
\be \label{eq:A_conv}
A = \sum_{j=1}^N a_j(t) = \frac{\hat{u}(t,0^+)}{-2 \pi i}  \quad \mbox{for $t \geq T_0$},
\ee
with some constant $A \in \C$, where the last equation follows from taking the Fourier transform of the right-hand side in \eqref{eq:u_sol_asy}. Since we must have $\hat{u}(t, 0^+) \neq 0$ by the discussion above, we conclude that $A \neq 0$ as well. Hence we deduce that
\be \label{eq:a_2}
\vec{a}(t) = O(t^{-1}) \vec{a}(t) + \vec{f}(t)
\ee
where $\vec{f}(t) = (f_2(t), \ldots, f_N(t))$ is given by
$$
f_k(t) =  \frac{-iA}{(\lambda_k+o(1)) (z_k(t) - z_1(t))} = \frac{-iA}{2 \lambda_k^2 t} + O(t^{-2}),
$$
thanks to the fact that $z_k(t)-z_1(t) = 2 \lambda_k t + O(1)$ by Lemma \ref{lem:z_largetime} and $\lambda_1=0$. In view of \eqref{eq:a_2}, we conclude that \eqref{eq:a_asymp} holds with the constant $\alpha = \frac{-iA}{2} \neq 0$.

Finally, by the conservation law \eqref{eq:A_conv} together with \eqref{eq:a_asymp} we immediately find
\be
\lim_{t \to +\infty} a_1(t) = A \neq 0 .
\ee

\medskip
\textbf{Step 2.} For $t \geq T_0$, we write
\be
u(t,x) = \sum_{j=1}^N \phi_j(t,x) \quad \mbox{with} \quad \phi_j(t,x) = \frac{a_j(t)}{x-z_j(t)}.
\ee
We note that
$$
\hat{\phi}_j(t,\xi) = -2 \pi i a_j(t) \mathrm{e}^{-i z_j \xi} \mathds{1}_{\xi \geq 0} .
$$
By Plancherel, we find
\begin{align*}
\langle \phi_j(t), \phi_k(t)\rangle_{H^s} & = -i a_j(t) \ov{a}_k(t) \int_0^\infty \mathrm{e}^{-i (z_j(t) - \ov{z}_k(t))\xi} \langle \xi \rangle^{2s} \, d \xi
\end{align*}
To estimate this expression for $j \neq k$, we observe, by Lemma \ref{lem:z_largetime}, that $\omega_{jk}(t):= z_j(t) - \ov{z}_k(t)$ satisfies $\mathrm{Im} \, \omega_{jk}(t) < 0$ as well as
$$
\omega_{jk}(t) = 2(\lambda_j-\lambda_k) t + O(1) \sim t \quad \mbox{as} \quad t \to +\infty,
$$
since $\lambda_j \neq \lambda_j$ when $j \neq k$. Integrating by parts sufficiently many times depending on $s>0$, we deduce that
$$
\left | \int_0^\infty \mathrm{e}^{-i \omega_{jk}(t)\xi} \langle \xi \rangle^{2s} \, d \xi \right | \lesssim_s \frac{1}{|\omega_{jk}(t)|} \sim \frac{1}{t} \quad \mbox{for $t \geq T_0$},
$$
provided that $j \neq k$. If recall the bounds for $a_1(t), \ldots, a_N(t)$ derived in \textbf{Step 1} above, we can conclude
$$
\left | \langle \phi_j(t), \phi_k(t) \rangle_{H^s} \right | \lesssim_s \frac{1}{t^2} \quad \mbox{for $t \geq T_0$ and $j \neq k$}.
$$

Next, we consider the case $j = k$. This yields
\begin{align*}
\langle \phi_j(t), \phi_j(t) \rangle_{H^s} & \sim  |a_j(t)|^2 \int_0^\infty \mathrm{e}^{2 \mathrm{Im}(z_j(t)) \xi} (1+ |\xi|^{2s})  d \xi \\
& = |a_j(t)|^2 \left ( \frac{1}{2 |\mathrm{Im} \, z_j(t)|} + \frac{C_s}{2  |\mathrm{Im} \, z_j(t)|^{1+2s}} \right )
\end{align*}
with the constant $C_s = \int_0^\infty \mathrm{e}^{-y} y^{2s} \, dy>0$. By combining the estimates for the coefficients $\{ a_1(t), \ldots, a_N(t) \}$ from \textbf{Step 1} and the poles $\{ z_1(t), \ldots, z_N(t) \}$ from Lemma \ref{lem:z_largetime} we finally obtain
$$
\langle \phi_j(t), \phi_j(t) \rangle_{H^s} \simeq_s \begin{dcases*} 1 & for $j=1$, \\ t^{4s} & for $j=2, \ldots, N$, \end{dcases*}
$$
for all times $t \geq T_0$. In summary, we conclude 
$$
 t^{4s} + t^{-2} \lesssim_s \| u(t) \|_{H^s}^2 = \sum_{j=1}^N \langle \phi_j(t), \phi_j(t) \rangle + \sum_{j \neq k}^N \langle \phi_j(t), \phi_k(t) \rangle \lesssim_s t^{4s} + t^{-2}.
$$
This proves \eqref{eq:u_sob} and completes the proof of Theorem \ref{thm:growth_intro}. \hfill $\qed$

\begin{appendix}

\section{Definition of the Lax Operator for $u \in L^2_+(\R)$}

\label{app:lax}

In view of the $L^2$-criticality of (CM-DNLS), it is worthwhile giving a definition of $L_u$ via quadratic forms if we only assume that $u \in L^2_+(\R)$ holds. We start with the following basic estimate.

\begin{lem}\label{lem:GN_basic}
For every $u\in L^2_+(\R )$ and $f\in H^{\frac 12}_+(\R )$, we have $T_{\overline u}f\in L^2_+(\R )$ with
$$\Vert T_{\overline u}f\Vert _{L^2}^2\leq \frac 1{2\pi} \Vert u\Vert_{L^2}^2\inner {Df} f\ .$$
\end{lem}

\begin{proof}
Applying the Fourier transformation, we have
$$\widehat{T_{\overline u}f}(\xi )=\int_0^{+\infty} \hat f(\xi +\eta )\overline{\hat u(\eta )}\, \frac{d\eta}{2\pi}\ .$$ 
Thus, by the Cauchy--Schwarz inequality,
\begin{eqnarray*}
\int_0^\infty |\widehat{T_{\overline u}f}(\xi )|^2\, d\xi &\leq &\int_0^{+\infty}\int_0^{+\infty} |\hat f(\xi +\eta )|^2\, \frac{d\xi \, d\eta}{2\pi}\  \int_0^{+\infty}|\hat u(\eta )|^2\, \frac{d\eta}{2\pi}\\
 &\leq & \int_0^{+\infty} \zeta |\hat f(\zeta)|^2\, \frac{d\zeta }{2\pi}\ \int_0^{+\infty}|\hat u(\eta )|^2\, \frac{d\eta}{2\pi}\ ,
\end{eqnarray*}
and the claim follows from Plancherel's theorem.
\end{proof}

Next, we define $L_u$ for given $u \in L^2_+(\R)$ via a densely defined quadratic form on $L_+^2(\R)$ as follows. For $f, g \in H^{1/2}_+(\R)$, we set
$$\mathcal{Q}_u(f,g)=\inner {Df} g - \inner {T_{\overline u}f} {T_{\overline u}g}  \ .$$
We claim that, for every $\e >0$, there exists a constant $C_\e (u)>0$ such that 
\be \label{refined}
\forall h\in H^{1/2}_+(\R )\ ,\ \Vert T_{\overline u}h\Vert_{L^2}\leq \e {\inner {Dh} h}^{1/2} +C_\e(u)\Vert h\Vert_{L^2}\ .
\ee
Indeed, for every $\lambda >0$, set 
$$u_{<\lambda }:={\bf 1}_{[0,\lambda [}(D)u\ ,\ u_{\geq \lambda }:={\bf 1}_{[\lambda ,+\infty [}(D)u\ .$$
Then $\Vert u_{\geq \lambda }\Vert_{L^2}\to 0$ as $\lambda \to +\infty $, while 
$$\Vert u_{<\lambda }\Vert_{L^\infty}\leq \left ( \frac{\lambda}{2\pi} \right )^{1/2}\Vert u\Vert _{L^2}\ .$$
Choose $\lambda =\lambda (\e, u)$ such that $(2\pi)^{-1/2} \Vert u_{\geq \lambda}\Vert_{L^2}\leq \e $. Then, by Lemma \ref{lem:GN_basic},
$$\Vert T_{\overline u}h\Vert _{L^2}\leq \Vert T_{\overline u_{\geq \lambda}}h\Vert _{L^2}+\Vert T_{\overline u_{<\lambda}}h\Vert _{L^2}\leq \e {\inner {Dh} h}^{1/2}+\Vert u_{<\lambda }\Vert_{L^\infty}\Vert h\Vert_{L^2}\ ,$$
and \eqref{refined} follows.
Applying \eqref{refined}, we obtain
\be \label{equivQ}
\inner {Df} f \geq \mathcal{Q}_u(f,f)\geq (1-2\e ^2)\inner {Df} f -2C_\e (u)^2\Vert f\Vert_{L^2}^2\ .
\ee
Choosing $\e $ small enough, we find a constant $K=K(u)>0$ such that 
$$\tilde{\mathcal{Q}}_u(f,g):=\mathcal{Q}_u(f,g)+K(u)\inner f g$$
is an inner product on $H^{1/2}_+(\R )$, defining a norm which is equivalent to the standard one. Then we just define $${\rm dom}(L_u) =\{ f\in H^{1/2}_+(\R ): \mbox{$\exists C>0$ s.\,t.~$|\mathcal{Q}_u(f,g)|\leq C\Vert g\Vert_{L^2}$ for $g\in H^{1/2}_+(\R )$}\}$$
and 
$$\inner {L_u(f)} g=\mathcal{Q}_u(f,g)\  \quad \mbox{for $f\in {\rm dom}(L_u)$ and $g\in H^{1/2}_+(\R )$} \ .$$
Then the standard theory of quadratic forms (see \cite{ReSi-80}) implies that ${\rm dom}(L_u)$ is dense in $H^{1/2}_+(\R )$, hence in $L^2_+(\R )$, and that $L_u$ is self-adjoint and bounded below. Furthermore, using the quadratic form $\mathcal{Q}_u$, Lemma \ref{lem:firstbracket} and Proposition \ref{prop:L_basic} extend easily to the case $u\in L^2_+(\R )$.

\section{Variational Properties of $E(u)$}\label{varE}

We recall the energy functional 
$$
E(u) = \frac{1}{2} \int_{\R} |\pt_x u - i \Pi_+(|u|^2) u|^2 \, dx,
$$
where we allow for general $u \in H^1(\R)$, which are not necessarily in the Hardy-Sobolev space $H^1_+(\R)$. It is elementary to show that $E : H^1(\R) \to \R_{\geq 0}$ is {\em weakly lower semi-continuous}, i.\,e., if $u_k \weakto u$ weakly in $H^1(\R)$ then
\be \label{eq:wlsc}
\liminf_{n \to \infty} E(u_n) \geq E(u).
\ee

\begin{lem}[Minimal Mass Bubble Lemma]\label{lem:MMB} Suppose $(v_n)_{n \in \N} \subset H^1(\R)$ is a sequence such that
$$
\sup_{n \geq 1} \| v_n \|_{L^2} < +\infty \quad \mbox{and} \quad \| \pt_x v_n \|_{L^2} = \mu \quad \mbox{for all $n \in \N$}.
$$ 
with some constant $\mu > 0$. In addition, we assume that
$$
\lim_{n \to \infty} E(v_n) = 0.
$$
Then it holds that
$$
\liminf_{n \to \infty} \| v_n \|_{L^2}^2 \geq \| \RR\|_{L^2}^2 = 2 \pi.
$$
Here  equality holds if and only if, after possibly passing to a subsequence, 
$$
v_n(x + x_n) \to \mathrm{e}^{i \theta} \lambda^{1/2} \RR(\lambda x)\quad \mbox{strongly in $L^2(\R)$}
$$ 
with some constants $\theta \in [0,2 \pi)$, $\lambda > 0$, and some sequence $x_n\in \R$.

\end{lem}

\begin{proof}
We will give a proof that is based on a compactness lemma in \cite{Li-83}. Alternatively, we could use more refined analysis with a profile decomposition \cite{Ge-98}. 

By rescaling $v_n \to \mu^{-1/2} v_n(\mu^{-1} \cdot)$, we can assume that $\| \pt_x v_n \|_{L^2} = 1$ for all $n \in \N$. From the triangle inequality and the form of $E(u)$, we find that
\begin{align*}
\| \Pi_+(|v_n|^2) v_n \|_{L^2} \geq \| \pt_x v_n \|_{L^2} - \| \pt_x v_n - i \Pi_+(|v_n|^2) v_n \|_{L^2} = 1 - \sqrt{2E(v_n)}.
\end{align*}
Since $E(v_n) \to 0$ by assumption, we deduce that
$$
1 \lesssim \| \Pi_+(|v_n|^2) v_n \|_{L^2} \lesssim \| \Pi_+(|v_n|^2) \|_{L^3} \| v_n \|_{L^6} \lesssim \| |v_n|^2 \|_{L^3} \| v_n \|_{L^6} \lesssim \| v_n \|_{L^6}^3, 
$$
by H\"older's inequality and the classical fact that $\Pi_+ : L^3(\R) \to L^3(\R)$ is bounded. Thus we have found  that
$$
\|v_n \|_{L^6} \geq C >0
$$
with some constant $C> 0$. On the other hand, by the fact $\sup_{n} \| v_n \|_{H^1} < +\infty$ and by Sobolev embeddings, we deduce that
$$
\| v_n \|_{L^2} \leq C_1 \quad \mbox{and} \quad \| v_n \|_{L^8} \leq C_2
$$
with some constants $C_1, C_2 > 0$. Thus, by applying the $pqr$-Lemma in \cite{FrLiLo-86}, we deduce that there exist constants $\eps > 0$ and $\delta > 0$ such that
$$
\mu( \{ x \in \R : |v_n(x)| > \eps \}) \geq \delta \quad \mbox{for all $n \in \N$},
$$
where $\mu$ denotes the Lebesgue measure on $\R$. Hence we can apply Lieb's compactness lemma in \cite{Li-83} to deduce that, after passing to a subsequence, there exists a sequence of translations $(y_n)_{n \in \N} \subset \R$ such that
$$
v_n(\cdot + y_n) \weakto v \quad \mbox{weakly in $H^1(\R)$}
$$
for some $v \in H^1(\R)$ with $v \not \equiv 0$. 

By translation invariance of the energy $E$, we can henceforth assume that $y_n=0$ for all $n$. Furthermore, the weak lower semi--continuity of $E$ implies that
$$
0 = \lim_{n \to \infty} E(v_n) \geq E(v) .
$$
On the other hand, we have $E(v) \geq 0$ in general and hence we conclude $E(v)=0$. By Lemma \ref{lem:R_uniq}, the equality $E(v)=0$ for $v \not \equiv 0$ holds if and only if 
$$
v(x) = \mathrm{e}^{i \theta} \lambda^{1/2} \RR(\lambda x+ y)
$$ 
with some constants $\theta \in [0,2 \pi)$, $\lambda > 0$ and $y \in \R$. Since $v_n \weakto v$ in $L^2$, the  weak lower semi--continuity of the $L^2$-norm implies that
$$
 \liminf_{n \to \infty} \| v_n \|_{L^2}^2 \geq \| v \|_{L^2}^2 = \| \RR \|_{L^2}^2 = 2\pi.
$$
Finally, we have equality (after passing to a subsequence if necessary) if and only if $v_n \to v$ strongly in $L^2(\R)$, which completes the proof.
\end{proof}

\section{Useful Identities and Gauge Transformation}\label{app:misc}

For sufficiently regular and decaying functions $v : \R \to \C$ and the Hilbert transform $\Hil$, we have the following identities:
\be \label{eq:Hil1}
\mathrm{Re} \langle x v, \Hil(v) \rangle = \frac{1}{2 \pi} \left \vert  \int_{\R} v \, dx \right \vert ^2\ ,
\ee
\be \label{eq:Hil2}
\mathrm{Re} \, \langle  \pt_x v, \Hil(|v|^2) v \rangle = -\frac{1}{2} \langle |v|^2, |D| |v|^2 \rangle,
\ee
\be \label{eq:Hil3}
 \langle v \Hil(|v|^2), v \Hil(|v|^2) \rangle = \frac{1}{3} \int_{\R} |v|^6 \, dx.
\ee
Let us check these identities. Using the Plancherel theorem, we have
$$ \mathrm{Re}\langle x v, \Hil(v) \rangle =-\frac{1}{2\pi}\int_0^\infty \mathrm{Re}[\ov{\hat v}(\xi)\pt_\xi \hat v(\xi)]\, d\xi +\frac{1}{2\pi}\int_{-\infty}^0 \mathrm{Re}[\ov{\hat v}(\xi)\pt_\xi \hat v(\xi)]\, d\xi =\frac{1}{2\pi}|\hat v(0)|^2\ ,$$
which is \eqref{eq:Hil1}.\\
Since $\Hil$ preserves real valued functions, we have
$$\mathrm{Re} \, \langle  \pt_x v, \Hil(|v|^2) v \rangle =\inner {\mathrm{Re}[\ov v\partial_xv]} {\Hil (|v|^2} =\frac 12\inner {\pt_x(|v|^2)} {\Hil (|v|^2}
=-\frac 12\inner {|v|^2}{\Hil \pt_x(|v|^2}\ ,$$
which leads to \eqref{eq:Hil2} since $\Hil\pt_x=|D|$. \\
Finally, setting $\rho :=|v|^2$, we have $\rho =\Pi_+\rho +\ov{\Pi_+\rho }$ and
$$\Vert v\Vert_{L^6}^6=\int_\R \rho^3\, dx =\int_\R (\Pi_+\rho +\ov{\Pi_+\rho} )^3\, dx =3\int_\R [(\Pi_+\rho)^2\ov{\Pi_+\rho} +\Pi_+\rho (\ov{\Pi_+\rho} )^2]\, dx \ ,$$
because  $H^s_\pm $ is preserved by the product (if $s$ is large enough) and the ranges of $\Pi_+$ and of $\Pi_-$ are orthogonal. We obtain
$$ \int_\R \rho ^3\, dx=3\int_\R |\Pi_+\rho |^2\rho \, dx =\frac 34 \int_\R (\rho^2+(H\rho )^2)\rho \, dx\ ,$$
whence 
$$\int_\R \rho ^3\, dx=3\int \rho (H\rho )^2\, dx \ ,$$
which is precisely \eqref{eq:Hil3}.

\subsection{Derivation of Pohozaev Identity}
In this paragraph we provide the details for the proof of Proposition \ref{prop:travel} when deriving the Pohozaev type identities.  We integrate \eqref{eq:S_om} against $x \pt_x \ov{S}$ over the compact interval $[-R,R]$. By taking the real part, we find 
$$
\mathrm{Re}  \int_{-R}^R \left ( -x \ov{S}' S'' -  x \ov{S}' (|D| |S|^2) S  + \frac{1}{4}x \ov{S}' |S|^4 S \right )  dx = \tilde{\omega} \mathrm{Re} \int_{-R}^R x \ov{S}' S \, dx.
$$
For the first term on the left-hand side, we find
\begin{align*}
\mathrm{Re} \int_{-R}^R x \ov{S}' S'' \, dx & = \frac{1}{2} \int_{-R}^R x \pt_x |S'|^2 \, dx = \frac{1}{2} x |S'(x)|^2 \Big |_{-R}^R - \frac{1}{2} \int_{-R}^R |S'(x)|^2 \, dx.
\end{align*}
In view of $|D|= \Hil \pt_x$ and integrating by parts, we obtain
\begin{align*}
& \mathrm{Re} \int_{-R}^R x \ov{S}' (|D| |S|^2) S \, dx = \frac{1}{2} \int_{-R}^R x \pt_x |S|^2 ( \Hil \pt_x |S|^2) \, dx   \\
&\to \frac{1}{2 \pi} \left ( \int_{-\infty}^{\infty} \pt_x |S|^2 \, dx \right )^2 = 0 \quad \mbox{as} \quad R \to +\infty,
\end{align*}
where we also used \eqref{eq:Hil1}. Next, we notice
$$
\mathrm{Re} \, \int_{-R}^R x \pt_x \ov{S} |S|^4 S \, dx = \frac{1}{6} x |S(x)|^6 \Big |_{-R}^R - \frac{1}{6} \int_{-R}^R |S|^6 \, dx
$$
$$
\tilde{\omega} \mathrm{Re} \int_{-R}^R x \pt_x \ov{S} S \, dx = \frac{\tilde{\omega}}{2} x |S(x)|^2 \Big |_{-R}^R  - \frac{\tilde{\omega}}{2} \int_{-R}^R |S(x)|^2 \, dx.
$$
Since $|S|^2 + |S'|^2 \in L^1(\R)$, there exists a sequence $R_n \to +\infty$ such that $x(|S(x)|^2+ |S'(x)|^2)  \to 0$ with $x=\pm R_n$ as $n \to \infty$ and we obtain
$$
\frac{1}{2} \int_{\R} |\pt_x S|^2 \,d x- \frac{1}{24} \int_{\R} |S|^6 \, dx = -\frac{\tilde{\omega}}{2} \int_{\R} |S|^2 \, dx.
$$

\subsection{Gauge Transformation and Pseudo-Conformal Law}

Let $s \geq 0$ be given. We consider the nonlinear map
\be
\Phi : H^s(\R) \to H^s(\R), \quad u(x) \mapsto u(x) \mathrm{e}^{-\frac{i}{2} \int_{-\infty}^x |u(y)|^2 \, dy}.
\ee
Clearly, we have mass preservation property $M[\Phi(u) ] = M(u)$ and it can be shown that $\Phi : H^s(\R) \to H^s(\R)$ is a diffeomorphism. We refer to the map $\Phi$ as the \textbf{gauge transform}. We remark that the map $\Phi$ (with different numerical factors in the exponential term) also plays an important role for the derivative NLS (DNLS).

Suppose $u(t) \in H^s(\R)$ (not necessarily restricted to $H^s_+(\R)$) solves \eqref{eq:NLS} on some time interval $I \subset \R$. Then $v(t)=\Phi(u(t)) \in H^s(\R)$ is found to solve the derivative type NLS equation:
\be \label{eq:nls_v}
i \pt_t v = -\pt_{xx} v - (|D| |v|^2) v + \frac{1}{4} |v|^4 v,
\ee
using that $\Pi_+ = \frac{1}{2}(1 + i \Hil)$ and $|D|= \Hil \pt_x$. We readily check that \eqref{eq:nls_v} has the conserved energy
\be
\widetilde{E}(v) =  \frac{1}{2} \int_{\R} |\pt_x v|^2 - \frac{1}{4}\int_{\R} |v|^2 (|D| |v|^2) + \frac{1}{24} \int_{\R} |v|^6 .
\ee
From identities \eqref{eq:Hil2} and \eqref{eq:Hil3}, we have
\be
\widetilde{E}(v) = \frac{1}{2} \int_{\R} |\pt_x v + \frac{1}{2} \Hil(|v|^2) v|^2 = \frac{1}{2} \int_{\R} |\pt_x u - i \Pi_+(|u|^2) u|^2 = E(u)
\ee
since $v = \Phi(u)$ and $\Pi_+ = \frac{1}{2}(1 + i \Hil)$.

Now, let $\Sigma = \{ u \in H^1(\R) : xu \in L^2(\R) \}$ denote the space of solutions with finite variance and energy (not necessarily restricted to the Hardy space).

\begin{lem}[Pseudo-Conformal Law] \label{lem:pseudo}
Suppose that $u \in C([-T,T]; H^1(\R))$ solves \eqref{eq:NLS} with $u(0) =u_0 \in \Sigma$. Then $u(t) \in \Sigma$ for all $t \in [-T,T]$ and we have
$$
\frac{d^2}{dt^2} \int_{\R} |x|^2 |u(t,x)|^2 \, dx = 16 E(u_0) \quad \mbox{for $t \in [-T,T]$}.
$$
As a consequence, it holds that
$$
8 t^2 E(\mathrm{e}^{i |x|^2/4t} u_0) = \int_{\R} |x|^2 |u(t,x)|^2 \,dx
$$
for all $t \in [-T,T]$ with $t \neq 0$.
\end{lem}

\begin{remarks*}
1) Recall that the ground state $\RR(x) = \frac{\sqrt{2}}{x+i} \in H^1_+(\R)$ does {\em not} belong to $\Sigma$. In fact, it can be shown that all multi-solitons for \eqref{eq:NLS} fail to have finite variance as well. 

2) In view of the non-negativity $E(u) \geq 0$, we see that the classical Zakharov--Glassey argument to prove existence for negative energy data (with finite variance) cannot be invoked for \eqref{eq:NLS}.
\end{remarks*}

\begin{proof}
Let $v(t) = \Phi(u(t))$ for $t \in [-T,T]$. We readily check that $\Phi : \Sigma \to \Sigma$ holds and, in particular, we have $v_0=v(0) \in \Sigma$. Recall that $v \in C([-T,T]; H^1(\R))$ solves \eqref{eq:nls_v}. By following arguments for standard-type NLS, see, e.\,g., \cite{CaHa-98}[Section 7.6], we can deduce that $v(t) \in \Sigma$ for all $t \in [-T,T]$. Moreover, by calculations analogous to $L^2$-critical NLS (see \cite{CaHa-98} again), we find the variance-virial identities
\be
\frac{d}{dt} \int_{\R} |x|^2 |v(t,x)|^2 \, dx = 4 \int_{\R} x \mathrm{Im} (\ov{v} \pt_x v) \, dx, 
\ee
\be
\quad \frac{d^2}{dt^2} \int_{\R} |x|^2 |v(t,x)|^2 \, dx = 16 \widetilde{E}_0,
\ee
with $v_0 = \Phi(u_0)$ and $\widetilde{E}_0 = \widetilde{E}(v_0)$. By integration in $t \in [-T,T]$, we obtain that the variance $V(t) = \int_{\R} |x|^2 |v(t,x)|^2 \, dx$ is given by
\be
V(t) = 8 \widetilde{E}_0 t^2 + A_0 t + V_0
\ee
with the constants $A_0 = 4 \int_{\R} x \mathrm{Im}(\ov{v}_0 \pt_x v_0) \, dx$ and $V_0=V(0)$. For $t \in [-T,T]$ with $t \neq 0$, we observe 
\begin{align*}
8 t^2 \widetilde{E}(\mathrm{e}^{i |x|^2/4t} v_0) & = 8t^2 \left ( \frac{1}{2} \int_{\R} | \pt_x( \mathrm{e}^{i |x|^2/4t} v_0) |^2 - \frac{1}{4} \int_{\R} |v_0|^2 (|D| |v_0|^2) + \frac{1}{24} \int_{\R} |v_0|^6 \right ) \\
& = 8 t^2 \widetilde{E}_0 + A_0 t + V_0  = V(t).
\end{align*}

Finally, we go back to the function $u=u(t,x)$. Here we note that $|u(t,x)| = |v(t,x)|$ and $\widetilde{E}(v_0) = E(u_0)$ since $v(t)= \Phi(u(t))$. This proves this first claim in Lemma \ref{lem:pseudo}. For the second statement, we remark that $\Phi$ commutes with multiplication by $\mathrm{e}^{i |x|^2/4t}$ for $t \neq 0$, i.\,e., we have $\Phi(\mathrm{e}^{i |x|^2/4t} u_0) = \mathrm{e}^{i |x|^2/4t} \Phi(u_0) = \mathrm{e}^{i |x|^2/4t} v_0$. Therefore $\widetilde{E}(\mathrm{e}^{i |x|^2|4t} v_0) = E(\mathrm{e}^{i|x|^2/4t} u_0)$. This completes the proof.
\end{proof}

\end{appendix}

\bibliographystyle{siam}
\bibliography{CMDNLSBib}

\end{document}